\documentclass[12pt,reqno]{amsart}
\usepackage{enumerate}
\usepackage{amsrefs}
\usepackage{amsfonts, amsmath, amssymb, amscd, amsthm, bm, cancel}
\usepackage{url}
\usepackage{graphicx}
\usepackage[
linktocpage=true,colorlinks,citecolor=magenta,linkcolor=blue,urlcolor=magenta]{hyperref}
\usepackage{multicol}
\usepackage{comment}
\usepackage[margin=1in]{geometry}
\usepackage{fancyhdr}

\usepackage{mathtools}
\mathtoolsset{showonlyrefs}

\newtheorem{thm}{Theorem}[section]

  \newtheorem{lem}{Lemma}[section]

 \newtheorem{prop}{Proposition}[section]
 \theoremstyle{definition}
 
  \newtheorem*{ack}{Acknowledgments}
 \theoremstyle{remark}
\newtheorem{rem}{Remark}[section]

 \numberwithin{equation}{section}

\allowdisplaybreaks

\newcommand{\f}{\left(}
\renewcommand{\r}{\right)}

\newcommand{\N}{\mathbb{N}}
\newcommand{\R}{\mathbb{R}}

\renewcommand{\H}{\mathbb{H}}

\renewcommand{\a}{\alpha}
\renewcommand{\b}{\beta}
\newcommand{\g}{\varphi}
\renewcommand{\d}{\delta}

\renewcommand{\k}{\kappa}

\newcommand{\s}{\sigma}

\newcommand{\metric}[2]{\ensuremath{\langle #1, #2\rangle}}  

\begin{document}
\title[Curvature estimates for semi-convex solutions]{Curvature estimates for semi-convex solutions of asymptotic Plateau problem in $\mathbb{H}^{n+1}$}

\author{Han Hong}
\address{Department of Mathematics and Statistics, Beijing Jiaotong University, Beijing 100044, P. R. China}
\email{\href{mailto:}{hanhong@bjtu.edu.cn}}

\author{RUIJIA ZHANG}
\address{Key Laboratory of Pure and Applied Mathematics School of Mathematical Sciences, Peking University, Beijing 100871, P. R. China}
\email{\href{mailto:zhangrj17@mails.tsinghua.edu.cn}{zhangrj@pku.edu.cn}}



\begin{abstract}
In this paper, we consider the asymptotic $\sigma_k$ Plateau problem in hyperbolic space. We establish $C^2$ estimates for semi-convex complete hypersurfaces  satisfying constant $\s_k$ curvature with a prescribed asymptotic boundary at the infinity for $2\leq k\leq n-2$ . The result is based on a new crucial concavity inequality derived for hessian equations.
\end{abstract}

\maketitle

\section{introduction}\label{sec:1}

Let 
\[\mathbb{H}^{n+1}=\{(z,z_{n+1})\in\mathbb{R}^{n+1}:\ z_{n+1}>0,z\in\mathbb{R}^n\}\]
with metric
\[g^{\mathbb{H}^{n+1}}=\frac{1}{z_{n+1}^2}g^{\mathbb{R}^n}\]
be the $n+1$ dimensional hyperbolic space, i.e., the half space model. Let $\partial_\infty \mathbb{H}^{n+1}$ denote the ideal boundary of $\mathbb{H}^{n+1}$ at infinity, which can be identified with $\mathbb{R}^n\times \{0\}$. We assume that $\Gamma$ is a closed $(n-1)$-dimensional embedded hypersurface in $\partial_\infty\mathbb{H}^{n+1}$, which can be viewed as the boundary of a smooth domain in $\mathbb{R}^n$. 

The asymptotic Plateau problem for $\sigma_k$ aims to find a complete hypersurface $\Sigma$ in $\mathbb{H}^{n+1}$ satisfying the following conditions:
\begin{equation}\label{equation1introduction}
\k[\Sigma]\in \Gamma_k\ \ \text{and}\ \ P_k(\kappa[\Sigma])=\sigma\in (0,1)
\end{equation}
with
\begin{equation}\label{boundaryconditionintroduction}
\partial\Sigma=\Gamma.
\end{equation}
Here, $\kappa[\Sigma]=(\kappa_1,\cdots,\kappa_n)$ denotes the hyperbolic principal curvatures of $\Sigma$. Then $P_k=\sigma_k/C_n^k$ is the $k$-th normalized elementary symmetric polynomial and $\Gamma_k$ denotes the $k$-th Garding cone, defined as
\[\Gamma_k:=\{\kappa\in\mathbb{R}^n:\ \sigma_k>0, \forall \ 1\leq j\leq k\}.\]

When $k=1$ or $n,$ this problem has been solved in \cite{GS1} and \cite{spruckgausscurvature}. When $k=n-1$, this problem has been solved by S. Lu \cite{Lu1}. The general case when $2\leq k\leq n-2$ is still open. If the hypersurface is assumed to be strictly (hyperbolic) convex, we can refer to \cites{guansurvey,GS2,guan2009jga,xiaoling2014jdg} by Guan-Spruck-Xiao for positive solutions and uniqueness results.

To solve the problem \eqref{equation1introduction}-\eqref{boundaryconditionintroduction}, one needs to find solutions to a fully nonlinear PDE with certain degeneracy. The idea is to use the method of continuity and to approximate the solution on compact subset of the domain bounded by $\Gamma$. The $C^0$ and $ C^1$ estimates are obtained by maximum principles with the help of geometric barriers and this has well-established for all $k$ in previous works \cites{guansurvey,GS2,guan2009jga,xiaoling2014jdg}. The curvature estimate can also be proven using a test function and the bound for principal curvatures depends on the lower bound of $u$. However, $u$ vanishes on the boundary $\Gamma$ and thus it prevents us applying diagonal argument to get the limit. The aim of this paper is to provide a curvature estimate for this problem that is uniformly independent of the lower bound of $u$ under semi-convex condition by using a new concavity inequality.

We say a hypersurface $M$ semi-convex if there exists a constant $A>0$ such that
\begin{align*}
\kappa_i(X)\geq -A,\quad 1\leq i\leq n, \quad \forall\ X\in M.
\end{align*}

We now state our main theorem.
\begin{thm}\label{main}
Let $\Gamma=\partial\Omega\times \{0\}$ has nonnegative (Euclidean) mean curvature, where $\Omega$ is a bounded smooth domain in $\mathbb{R}^n$ for $n\geq 4$. Suppose that $\s\in (0,1)$. For $2\leq k\leq n-2$, let $\Sigma$  be a complete semi-convex  hypersurface in $\mathbb{H}^{n+1}$ satisfying
\begin{align*}
P_{k}(\kappa)=\sigma,\quad \partial\Sigma=\Gamma
\end{align*}
with $\k\in \Gamma_k$.
Then we have 
\begin{align*}
    \max_{x\in \Sigma;i=1,2,...,n}|\k_i|\leq C
\end{align*}
where $C$ depends only on $n$, $k$, $\s$, $\Omega$ and $A$.
\end{thm}
\begin{rem}
    After completing this work, we found an interesting paper of Bin Wang \cite{Wang}, where
he completely resolved this problem for the cases when $k=2$ and $k=n-2$. Besides, he achieved the same curvature estimates under the assumption $\s_{k+1}>-A$ as stated in Theorem 1.9 of \cite{Wang}, and posed the question whether the results remain valid under the semi-convexity condition. We provide a positive answer to this question in Theorem \ref{main}.  
\end{rem}
The main difficulty of overcoming the curvature estimates is dealing with the third-order derivatives. The usual method is to use the concavity of $q_k$, or, more progressively, to use the concavity of $\f\frac{\s_k}{\s_l}\r^{\frac{1}{k-l}}=(q_k\cdots q_{l+1})^{\frac{1}{k-l}}$ and to throw out these terms directly. To prove this theorem, we derive a crucial concavity inequality based on Huisken-Sinestrari's observation \cite{HS} (see Lemma  \ref{sigma_k-1hessianqk}). According to \cite{HS}, when $\k\in \Gamma_k$, the concavity of $q_k$ can be deduced from the concavity of $q_{k-1}$ by induction. In particular, $-\s_1\partial^2_{\xi} q_2$ can be bounded below by the length of $\xi^{\perp}$ with respect to the vector $\k$. We observe that, by induction, the quantity of $-\partial^2_{\xi} q_k$ can be controlled by $\xi$. This observation plays a crucial role in controlling the third-order derivatives. 

This paper is organized as follows. In section \ref{sec:2}, we collect some known facts of asymptotic Plateau problem in $\mathbb{H}^{n+1}$, collect and prove some properties of the k-th elementary symmetric functions. In section \ref{sec:3}, we complete the proof of Theorem \ref{main}, which will be divided into several cases.
\begin{ack}
The authors would like to thank Prof. Haizhong Li for discussions and for the consistent encouragement provided. The second author is also indebted to Prof. Yuguang Shi for the enduring help and inspiration. The first author is supported by the Talent Fund of Beijing Jiaotong University No. 2024XKRC008. The second author is supported by the National Key R$\&$D Program of China 2020YFA0712800 and the China Postdoctoral Science Foundation No.2023M740108.
\end{ack}

\section{Preliminaries}\label{sec:2}
\subsection{Hypersurfaces in hyperbolic space}

In this section, we recall and collect some formulas of hypersurfaces in hyperbolic space. Let $\Sigma$ be a connected, oriented, complete hypersurface in $\mathbb{H}^{n+1}$ with asymptotic boundary at infinity. We will use the half-space model of $\mathbb{H}^{n+1}$. 

Let $X$ and $\nu$ be the position and the outward unit normal  of $\Sigma$  in $\mathbb{R}^{n+1}$ and $\mathrm{n}$ be the outer unit normal of $\Sigma$ in $\mathbb{H}^{n+1}$.
We denote by $u$ the height function of $\Sigma$ in $\mathbb{R}^{n+1}$, $e$ the unit vector field in the $x_{n+1}$ direction in $\mathbb{R}^{n+1}$ and $\cdot$ the Euclidean inner product. Then 
\begin{align*}
    u=X\cdot e.
\end{align*}
Let
$
    \nu^{n+1}=\nu\cdot e.
$
Take $\lbrace x_1,\cdots,x_n\rbrace$ to be a local normal coordinate system on $\Sigma$ around a point $p$. We denote by $g_{ij}=g(X_i,X_j)$ the induced metric and $h_{ij}=-g(\nabla_{X_i}X_j,\mathrm{n})$ the second fundamental form of $\Sigma$ in $\mathbb{H}^{n+1}$, $\tilde{g}_{ij}=g(X_i,X_j)$ the induced metric and $\tilde{h}_{ij}=-g(\tilde{\nabla}_{X_i}X_j,\nu)$ the second fundamental form of $\Sigma$ in $\mathbb{R}^{n+1}$, where $\nabla$ and $\tilde{\nabla}$ denote the Levi-Civita connection of $\H^{n+1}$ and $\R^{n+1}$ respectively. We define $f_{ij}=\nabla_{ij} f$ and $f_{ijk}=\nabla_{ijk} f$ for any $f\in C^4(\Sigma)$ below.

The Weingarten matrix is regarded as $\mathcal{W}=\lbrace h_i{}^j\rbrace=\lbrace h_{ik}g^{kj}\rbrace$, where $\lbrace g^{ij}\rbrace$ is the inverse matrix of $\lbrace g_{ij}\rbrace$. The principal curvatures $\k=(\k_1,\cdots,\k_n)$ of $M^n$ are eigenvalues of $\mathcal{W}$. Let $f(\k)$ be a symmetric function of the principal curvatures $\k = (\k_1, \k_2, \cdots, \k_n)$. There exists a function $\mathcal{F}(\mathcal{W})$ defined on the Weingarten matrix, such that $F (\mathcal{W})=f(\k)$.
Since $h_i{}^j = \sum_k h_{ik} g^{kj}$, $\mathcal{F}$ can be viewed as a function $\hat{\mathcal{F}}(h_{ij}, g_{ij})$ defined on the second fundamental form $\lbrace h_{ij}\rbrace$ and the metric $\lbrace g_{ij}\rbrace$. In the subsequent article, we denote
\begin{align*}
\dot{ \mathcal{F} }^{pq}(\mathcal{W}): =&\frac{\partial \hat{\mathcal{F}}}{\partial h_{pq}}(h_{ij}, g_{ij})=g^{pl}\frac{\partial \mathcal{F}}{\partial h_q^l },\quad
\ddot{\mathcal{F}}^{pq,rs}(\mathcal{W}): =\frac{\partial ^2 \hat{\mathcal{F}} }{\partial h_{pq}\partial h_{rs}}(h_{ij}, g_{ij})=g^{lp}g^{mr}\frac{\partial^2 \mathcal{F}}{\partial h_q^l\partial h_s^m}.
\end{align*}
Suppose that $\lbrace h_{ij} \rbrace$ is diagonal by choosing the appropriate orthonormal basis at one point. Then we define
\begin{align*}
\dot{f}^{p}: =\frac{\partial f}{\partial\k_p}(\k)=\dot{\mathcal{F}}^{pp}(\mathcal{A}),\quad\ddot{f}^{pq}: =\frac{\partial^2 f}{\partial\k_p\partial\k_q}(\k)=\ddot{\mathcal{F}}^{pp,qq}(\mathcal{A}).
\end{align*}
We collect three Lemmas corresponding to the geometric quantities of $\Sigma$ in $\R^{n+1}$ and $\H^{n+1}$ respectively, which can be found in \cite{GS2}, \cite{guansurvey} and \cite{xiaoling2014jdg}.
\begin{lem}
The relation between metric tensor is
    \begin{align}
        g_{ij}=\frac{\tilde{g}_{ij}}{u^2}.
    \end{align}
    The relation between Christoffel symbols is
\begin{align}\label{gamma}
\Gamma^k_{ij}=\tilde{\Gamma}_{ij}^k-\frac{1}{u}(u_i\delta_{kj}+u_j\delta_{ki} -\tilde{g}^{kl}u_l  \tilde{g}^{ij} ).
 \end{align}
\end{lem}

\begin{lem}
    \begin{align}
        1-|\tilde{\nabla} u|_{\tilde{g}}^2=(\nu^{n+1})^2,\quad 
        \tilde{\nabla}_{ij}u=-\tilde{h}_{ij}\nu^{n+1}.
    \end{align}
    \begin{align}
        \nabla_{ij}u=-\tilde{h}_{ij}\nu^{n+1}+2\frac{u_iu_j}{u}-\frac{u_k^2}{u}\delta_{ij}.
    \end{align}
\end{lem}
\begin{proof}
At a point $p$, we have
    \begin{align}
        \nabla_{ij}u=&\tilde{\nabla}_{ij} u+\tilde{\Gamma}_{ij}{}^ku_k\nonumber\\
        =&\tilde{\nabla}_{ij} u+2\frac{u_iu_j}{u}-\frac{u_k^2}{u}\delta_{ij}\nonumber\\
        =&-\tilde{h}_{ij}\nu^{n+1}+2\frac{u_iu_j}{u}-\frac{u_k^2}{u}\delta_{ij}.
    \end{align}
\end{proof}
\begin{lem}
The relation between the second fundamental form is
\begin{align}
      h_{ij}=-\frac{\tilde{h}_{ij}}{u}+\frac{\nu^{n+1}}{u^2}\tilde{g}_{ij}.
\end{align}
The relation between the Weingarten matrix and principal curvatures are
 \begin{align}\label{hik}
     h_i{}^k=-u\tilde{h}_i{}^k+\nu^{n+1}\delta_i{}^k,\quad \k_i=u\tilde{\k}_i+\nu^{n+1}, \quad i=1,\cdots,n,
 \end{align}
\end{lem}
Then we collect two Lemmas about $\nu_{n+1}$ referring to \cite{GS2} and \cite{Lu1}.
\begin{lem}\label{bnu}
For $n\geq 4$, let $\Sigma$ be a graph satisfying \eqref{equation1introduction} over a domain $\Omega$ with nonnegative (Euclidean) mean curvature boundary. Then
\begin{align*}
1\geq\nu^{n+1}> a_1>0
\end{align*}
 for some positive constant $a_1$ depending only on $n,\Omega$ and $\sigma$.
\end{lem}

\begin{lem}\label{nu,l}
\begin{equation}\label{fnu}
    \dot{F}^{ij}\nabla_{ij}\nu^{n+1} =-\f\dot{F}^{ij}h_{kj}h_i{}^k+\dot{F}^{ii}\r\nu^{n+1}+2\frac{\dot{F}^{ij}u_i\nu^{n+1}_j}{u}+kF(1+(\nu^{n+1})^2).
\end{equation}
\end{lem}

\begin{proof}
\begin{align}\label{nu}
\tilde{\nabla}_i\nu^{n+1}=\tilde{h}_i^ku_k,
    \end{align}
   Combining \eqref{hik} and \eqref{nu}, we obtain at point $p$ that
   \begin{align}\label{nu2}
   \nabla_i\nu^{n+1}=(\nu^{n+1}\delta_i{}^k-h_i{}^k)\frac{u_k}{u}. 
   \end{align}
Multiplying \eqref{nu2} with $\dot{F}^{ij}u_i$, we have
   \begin{align}
\dot{F}^{ij}u_i\nu^{n+1}_j=\frac{\dot{F}^{ij}u_i u_j\nu^{n+1}}{u}-\frac{\dot{F}^{ij}h_j^k u_k u_i}{u}.
   \end{align}
   Derivating both sides of \eqref{nu2}, we have 
\begin{align}\label{nuij1}
       \nabla_{ij}\nu^{n+1}=&-\frac{h_i^k{}_j u_k}{u}+\frac{\nabla_j\nu^{n+1}u_i}{u}-\frac{u_{kj}(h_i{}^k-\nu^{n+1}\delta_i{}^k)}{u}-\frac{u_j\nabla_i\nu^{n+1}}{u}\nonumber\\
       =&-\frac{h_i^k{}_j u_k}{u}+\frac{\nabla_j\nu^{n+1}u_i}{u}-\frac{u_{kj}h_i{}^k}{u}+\frac{\nu^{n+1}u_{ij}}{u}-\frac{u_j\nabla_i\nu^{n+1}}{u}.
\end{align}
   Multiply both sides of \eqref{nuij1} by $\dot{F}^{ij}$, and sum over $i,j$. This gives 
   \begin{align}
    \dot{F}^{ij}\nabla_{ij}\nu^{n+1}
       =&-\dot{F}^{ij}\frac{h_i^k{}_j u_k}{u}-\dot{F}^{ij}\frac{u_{kj}h_i{}^k}{u}+\dot{F}^{ij}\frac{\nu^{n+1}u_{ij}}{u}\nonumber\\
=&\frac{\dot{F}^{ij}\tilde{h}_{kj}h_i{}^k\nu^{n+1}}{u}-2\frac{\dot{F}^{ij}h_i^lu_lu_j}{u^2}+\frac{\dot{F}^{ij}h_{ij}u_k^2}{u^2}\nonumber\\
&-\frac{\dot{F}^{ij}\tilde{h}_{ij}(\nu^{n+1})^2}{u}
+2\frac{\dot{F}^{ij}u_iu_j\nu^{n+1}}{u^2}-\frac{\dot{F}^{ii}u_k^2\nu^{n+1}}{u^2}\nonumber\\
       =&\dot{F}^{ij}(\nu^{n+1}\delta_{kj}-h_{kj})h_i{}^k\nu^{n+1}+2\frac{\dot{F}^{ij}u_i\nu^{n+1}_j}{u}+kF(1-(\nu^{n+1})^2)\nonumber\\
       &-\f\dot{F}^{ii}(\nu^{n+1})^3-kF(\nu^{n+1})^2\r-\dot{F}^{ii}\f\nu^{n+1}-(\nu^{n+1})^3\r\nonumber\\
       =&-\f\dot{F}^{ij}h_{kj}h_i{}^k+\dot{F}^{ii}\r\nu^{n+1}+2\frac{\dot{F}^{ij}u_i\nu^{n+1}_j}{u}+kF(1+(\nu^{n+1})^2).
\end{align}
\end{proof}
\subsection{Symmetric function}
Here we state some algebraic properties of the elementary symmetric function $\s_m(\k)$, $m=1,\cdots,n$, where $\kappa=(\kappa_1,\cdots,\kappa_n)$. We recall $\s_m(\k)$ defined by
	\begin{equation*}
	\s_m(\k)=\sum_{1\leq i_1<\cdots<i_m \leq n}\k_{i_1}\cdots \k_{i_m}, \quad m=1,2, \cdots,n.
	\end{equation*}
	The Garding cone $\Gamma_m^+$ is an open symmetric convex cone in $\mathbb{R}^n$ with vertex at the origin, given by
	\begin{align*}
	\Gamma_m^{+}=\lbrace\f \kappa_1,\cdots, \kappa_n\r\in\mathbb{R}^n |\s_j(\kappa)>0,  \forall j= 1,\cdots,m\rbrace.
	\end{align*}
	Clearly, 
	\begin{align*}
	\Gamma_n^+\subset\cdots\subset {\Gamma}_m ^+\subset\cdots\subset\Gamma_1^+.
	\end{align*}
	In particular, $\Gamma^+ = \Gamma_n^+$ is called the positive cone,
	\begin{align*}
	\Gamma^+=\lbrace\f \kappa_1,\cdots, \kappa_n\r\in\mathbb{R}^n |\kappa_1>0,\cdots,\kappa_n>0\rbrace.
	\end{align*}
	Always assume $\k_1\geq \k_2\geq\cdots\geq \k_n$. We collect some properties of $\s_k(\kappa)$ where $\k\in\Gamma_k$, referring to \cite{Ben}, \cite{HS}, \cite{RW} and \cite{Wang09}.
\begin{lem}\label{f11k12}
$\sum_{p=1}^n\dot{\s}_k^{pp}\k_p^2=\s_1\s_k-(k+1)\s_{k+1} $ and $\sum_{p=1}^n \dot{\s}_k^{pp}=(n-k+1)\sigma_{k-1}.$
\end{lem}
 \begin{lem}\label{lem'}
    $ \frac{\partial \s_k}{\partial \k_n}\geq \cdots\geq \frac{\partial \s_k}{\partial \k_1}\geq0.$
 \end{lem}
 \begin{lem}\label{fk1}
$\dot{\s_k}^{11}\k_1\geq \frac{k}{n}\s_k$
 \end{lem}
 \begin{lem}\label{con}
   $\frac{\s_k}{\s_{k-1}}$ is concave.
 \end{lem}
 \begin{lem}\label{k1}
    For any $1\leq s<k$, $\s_s>\k_1\k_2\cdots\k_s$.
 \end{lem}
 \begin{lem}\label{lemn}
   $\k_k>0$ and $|\k_n|<(n-k)\k_k.$  
 \end{lem}
 \begin{proof}
 From Lemma \ref{lem'}, we have
 \begin{align*}
      \frac{\partial^{k-1} \s_k}{\partial\k_1\cdots\partial \k_{k-1}}=\k_k+\cdots+\k_n>0,
 \end{align*}
Suppose only $\kappa_n$ is negative, then $(n-k)\k_k>|\k_n|$.
 \end{proof}

\begin{lem}
    If $\s_{k+1}>-A$ for any positive constant A, then $\k$ is semi-convex, i.e. $\k_{\min}>-C(n,k,A)$.
\end{lem}
\begin{proof}
There exists constant $C>0$, such that
\begin{align}\label{sk^2}
    \frac{\partial \s_k}{\partial \k_n} \k_n^2 \leq\sum_i \frac{\partial \s_k}{\partial \k_i} \k_i^2 = \s_1 \s_k -(k+1) \s_{k+1}<C\s_1+(k+1)A.
\end{align}
Then by Lemma \ref{lem'} we have 
\begin{align*}
    \sum_i \frac{\partial \s_k}{\partial \k_i} \k_n^2 <nC\s_1+n(k+1)A.
\end{align*}
Combining $(n-k+1)\s_{k-1}=\sum_i \frac{\partial \s_k}{\partial \k_i}$ and Lemma \ref{k1}, Lemma \ref{lemn}, there exists constant $C(n,k)>0$
\begin{align*}
C(n,k)|\k_n|^k\k_1<Cn^2\k_1+n(k+1)A.
\end{align*}
Thus we get $|\k_n|<C(n,k,A)$ by Lemma \ref{lemn}, completing the proof.
\end{proof}
 
 \begin{lem}\label{upperboundforkappak}
     If $\k_n>-A$, then $\k_k<C(\s_k, n, A).$
 \end{lem}
 \begin{proof}
     Assume that $\k_{k-1}>A$, otherwise $\k_k\leq A$, then we are done. Since 
     \begin{align*}
        \s_k=\sum_{1\leq i_1\leq i_2...\leq i_k\leq n}\k_{i_1}\k_{i_2}\cdots \k_{i_k}>0,
     \end{align*}
     We have 
     \begin{align*}
         \k_1\cdots\k_{k-1}\k_k\leq \s_k-C(n,k)\k_1\cdots\k_{k-1}\k_n.
     \end{align*}
     Thus
\begin{align}
          \k_k\leq & 2\f\frac{\sigma_k}{\k_1}\r^{\frac{1}{k-1}}+2C(n,k)|\k_n|\label{kk,1}\\
          \leq &2\f \frac{n\sigma_k}{\s_1}\r^{\frac{1}{k-1}}+2C(n,k)A\nonumber\\
          \leq &c(n,k)\sigma_k^{\frac{1}{k}}+C(n,k,A).\nonumber
     \end{align}
  Here we use Lemma \ref{k1} in the second inequality, the last inequality follows from Maclaurin's inequality. Thus there exists some constant $C>0$, such that $\k_k\leq C$, which depends on the bounds of $\s_k$, $n$ and $A$. 
 \end{proof}

We now state a Lemma by Lu (\cite{Lu1}*{Lemma 3.1}),
which we'll need in the proof of our Theorem \ref{main}. For convenience, we provide a concise proof here. We also point out that we will denote by $C$, $c$ or $C_i$ for $i=1,2,\dots,n$ some positive constants in subsequent, which may change line by line.
\begin{lem}\label{lu}
   Assume $\lbrace\k_i\rbrace\in\Gamma_k$ and $\k_1\geq\k_2\cdots\geq \k_n$ and there exists some constant $A>0$ such that $\k_n>-A$. For any $\epsilon>0$,  there exists a large $N:=N(\epsilon)>\max\lbrace1,C(\s_k, n, A)\rbrace$ such that if $k_1\geq N^{2^{k-1}}$, then the following inequality holds
    \begin{align}\label{lu,iq}
        -\sum_{p\neq q}\frac{\s_k^{pp,qq}\xi_p\xi_q}{\s_k}+\frac{(\sum_i\s_k^{ii}\xi_i)^2}{\s_k^2}+\sum_{i>1}\frac{\s_k^{ii}\xi_i^2}{\k_1\s_k}\geq (1-\epsilon)\frac{\xi_1^2}{\k_1^2},
    \end{align}
    where $\xi=(\xi_1,\cdots,\xi_n)$ is an arbitrary vector in $\mathbb{R}^n$.
\end{lem}
Before proving Lemma \ref{lu}, we state the following proposition which is evident through an inductive proof and is crucial for our subsequent arguments. 
\begin{prop}\label{l,N}
    Assume $\lbrace\k_i\rbrace\in \Gamma_k$ and $\k_1\geq\k_2\cdots\geq \k_n$. For any constant $N>\max\lbrace1,C(\s_k, n, A)\rbrace$, if $\k_n>-A$ and $\k_1\geq N^{2^{k-1}}$, then there exists $1\leq l\leq k-1$ such that $\k_l\geq M^2$ and $\k_{l+1}\leq M$, where $M=N^{2^{k-l-1}}.$
\end{prop}
\begin{rem}
    $C(\s_k, n, A)$ in Proposition \ref{l,N} is the one derived in Lemma \ref{upperboundforkappak}. 
\end{rem}

\begin{proof}[Proof of Lemma \ref{lu}]
    We assume that there exists $M>\max \lbrace C(\s_k, n, A),1\rbrace$, such that $\k_l\geq M^2$ and $\k_{l+1}\leq M$ for some $1\leq l\leq k-1$. We note that $\log\s_l$ is concave. From Lemma \ref{con}, we have for any vector $\xi=(\xi_1,\cdots,\xi_n)$ in $\mathbb{R}^n$
    \begin{equation}\label{xi1}
   \begin{aligned}
-\partial^2_{\xi}\log\s_l=&-\partial^2_{\xi}\left(\log q_l+\log q_{l-1}+\cdots+\log q_2+\log\s_1\right)\\
=&-\frac{\partial^2_{\xi}q_l}{q_l}+(\partial_{\xi}\log q_l)^2-\frac{\partial^2_{\xi}q_{l-1}}{q_{l-1}}+(\partial_{\xi}\log q_{l-1})^2+\cdots+(\partial_{\xi}\log\s_1)^2
\geq 0
    \end{aligned}     
    \end{equation}
where we denote by $q_l=\frac{\s_l}{\s_{l-1}}$.
By direct computation, we have 
\begin{equation}\label{sl}
    \begin{aligned}
     -\partial^2_{\xi}\log\s_l=&-\frac{\partial^2_{\xi}\s_l}{\s_l}+\frac{(\partial_{\xi}\s_l)^2}{\s_l^2}\\
     =&\frac{(\dot{\s}_l^{11}\xi_1)^2}{\s_{l}^2}+\frac{2\sum_{\alpha\neq 1}(\dot{\s}_l^{11}\dot{\s}_l^{\a\a}-\s_l\ddot{\s}_l^{11,\a\a})\xi_1\xi_{\a}}{\s_{l}^2}-\partial^2_{\hat{\xi}}\log\s_l\\
     =&\frac{(\dot{\s_l}^{11}\xi_1)^2}{\s_l^2}+\frac{2\sum_{\a\neq 1}\left(\s_{l-1}(\k|1,\a)^2-\s_l(\k|1,\a)\s_{l-2}(\k|1,\a)\right)\xi_1\xi_{\a}}{\s_l^2}-\partial^2_{\hat{\xi}}\log\s_l.
    \end{aligned}
\end{equation}
Here $\hat{\xi}=(0,\xi_2,\dots,\xi_n)$. Since
\begin{equation}\label{sk,con}
 \begin{aligned}
      -\sum_{p \neq q}\frac{\s_k^{pp,qq}\xi_p\xi_q}{\s_k}+\frac{(\sum_i\s_k^{ii}\xi_i)^2}{\s_k^2}=-\partial^2_{\xi}\log\s_k=&-\partial^2_{\xi}\log q_k\cdots-\partial^2_{\xi}\log q_{l+1}-\partial^2_{\xi}\log \s_l\\
\geq&-\partial^2_{\xi}\log\s_l.
    \end{aligned}    
\end{equation}
Due to \eqref{xi1}, $-\partial^2_{\hat{\xi}}\log\s_l\geq 0$. Then by \eqref{sk,con}, we only need to estimate the first two terms of RHS in \eqref{sl}.

Combining Lemma \ref{k1} we have
\begin{align*}
\k_1\k_2\cdots\k_l\leq\s_l\leq \k_1\k_2\cdots\k_l+C\k_1\k_2\cdots\k_{l-1}\k_{l+1}\leq(1+\frac{C_1}{M})\k_1\k_2\cdots\k_l
\end{align*}
 and 
 \begin{align*}
\dot{\s_l}^{11}\geq\k_2\k_3\cdots\k_l.
 \end{align*}
 Here $C_1>0$ depends on $n,k$. Thus we have 
 \begin{align}\label{lu1}
     \frac{(\dot{\s_l}^{11}\xi_1)^2}{\s_l^2}\geq (1-C_1M^{-1})\frac{\xi_1^2}{\k_1^2}.
 \end{align}
 Meanwhile, for any  $1<i\leq l$
   \begin{align*}
    \left|\s_{l-1}(\k|1,i)^2-\s_l(\k|1,i)\s_{l-2}(\k|1,i)\right|\leq \frac{C}{M^2}\f\k_2\k_3\cdots\k_l\r^2.
 \end{align*}
 Then 
\begin{align}\label{luj}
  \left|\frac{2\left(\s_{l-1}(\k|1,i)^2-\s_l(\k|1,i)\s_{l-2}(\k|1,i)\right)\xi_1\xi_{i}}{\s_l^2}\right|\leq2\frac{C|\xi_1\xi_i|}{M^2\k_1^2} 
\end{align}
Due to Lemma \ref{fk1}, for any $1< i\leq n$ we have
\begin{align}\label{fii}
  \s_k^{ii}\geq\s_k^{11} \geq\frac{k\s_k}{n\k_1}.
\end{align}
Then by \eqref{luj} and \eqref{fii}
\begin{align}\label{lu2}
  \left|\frac{2\left(\s_{l-1}(\k|1,i)^2-\s_l(\k|1,i)\s_{l-2}(\k|1,i)\right)\xi_1\xi_{i}}{\s_l^2}\right|\leq&\frac{nC^2\xi_1^2}{kM^4\k_1^2}+\frac{k\xi_i^2}{n\k_1^2}  \nonumber\\
  \leq &C_2M^{-4}\frac{\xi_1^2}{\k_1^2}+\frac{\dot{\s}_k^{ii}\xi_i^2}{\s_k\k_1}
\end{align}
where we use the Cauchy-Schwartz inequality.

For any $p\geq l+1$, 
 \begin{align*}
    \left|\s_{l-1}(\k|1,p)^2-\s_l(\k|1,p)\s_{l-2}(\k|1,p)\right|\leq C\f\k_2\k_3\cdots\k_l\r^2.
 \end{align*}
Then
\begin{align}\label{lup}
  \left|\frac{2\left(\s_{l-1}(\k|1,p)^2-\s_l(\k|1,p)\s_{l-2}(\k|1,p)\right)\xi_1\xi_{p}}{\s_l^2}\right|\leq2\frac{C|\xi_1\xi_p|}{\k_1^2}.
\end{align}
Recall
 \begin{align}
\dot{\s_k}^{11}=\kappa_p\ddot{\s_k}^{11,pp}+\s_{k-1}(\k|1,p)
 \end{align}
 \begin{align}  \dot{\s_k}^{pp}=\kappa_1\ddot{\s_k}^{11,pp}+\s_{k-1}(\k|1,p)   
 \end{align}        
If $\s_{k-1}(\k|1,p)\leq0$, since $M>1$,
\begin{align}\label{sk-1<0}
M^{-1}\dot{\s_k}^{pp}=M^{-1}\kappa_1\ddot{\s_k}^{11,pp}+M^{-1}\s_{k-1}(\k|1,p)\geq\dot{\s_k}^{11}.
\end{align}
If $\s_{k-1}(\k|1,p)>0$, using the Maclaurin's inequality, we have 
\begin{align}\label{maciq}
    \s_{k-1}(\k|1,p)^{\frac{k-2}{k-1}}\leq\s_{k-2}(\k|1,p).
\end{align}
Since
\begin{align*}
   \s_{k-1}(\k|1,p)\leq C\k_2\k_3\cdots\k_l\k_{l+1}^{k-l}\leq CM^{-(k-l)}\k_1^{k-1},
\end{align*}
multiplying it with \eqref{maciq} we obtain
\begin{align*}
 \s_{k-1}(\k|1,p)\leq (CM^{-(k-l)})^{\frac{1}{k-1}}\k_1\ddot{\s_k}^{11,pp}.   
\end{align*}
Then
\begin{align}\label{s_k-1>0}
(M^{-1}+C^{-\frac{1}{k-1}}M^{-\frac{k-l}{k-1}})\dot{\s_k}^{pp}\geq(M^{-1}+C^{-\frac{1}{k-1}}M^{-\frac{k-l}{k-1}})\kappa_1\ddot{\s_k}^{11,pp}\geq \dot{\s_k}^{11}.
\end{align}
Using \eqref{fii}, \eqref{sk-1<0} and \eqref{s_k-1>0}, we have that for any $p\geq l+1,$
\begin{align}
    C_3M^{-\frac{k-l}{k-1}}\dot{\s_k}^{pp}\geq (M^{-1}+CM^{-\frac{k-l}{k-1}})\dot{\s_k}^{pp}\geq\frac{k\s_k}{n\k_1}
\end{align}
By using the Cauchy-Schwartz inequality, \eqref{lup} becomes
\begin{align}\label{lup2}
  \left|\frac{2\left(\s_{l-1}(\k|1,p)^2-\s_l(\k|1,p)\s_{l-2}(\k|1,p)\right)\xi_1\xi_{p}}{\s_l^2}\right|\leq&\frac{nC^2C_3\xi_1^2}{kM^{\frac{k-l}{k-1}}\k_1^2}+\frac{kM^{\frac{k-l}{k-1}}\xi_p^2}{nC_3\k_1^2}  \nonumber\\
  \leq &C_4M^{-\frac{k-l}{k-1}}\frac{\xi_1^2}{\k_1^2}+\frac{\dot{\s_k}^{pp}\xi_p^2}{\s_k\k_1}
\end{align}
Combining \eqref{lu1}, \eqref{lu2} and \eqref{lup2}, we plug \eqref{sl} into \eqref{sk,con} and obtain
 \begin{equation}\label{lu,2}
     \begin{aligned}
      -\sum_{p\neq q}\frac{\s_k^{pp,qq}\xi_p\xi_q}{\s_k}+\frac{(\sum_i\s_k^{ii}\xi_i)^2}{\s_k^2}+\sum_{i>1}\frac{\s_k^{ii}\xi_i^2}{\k_1\s_k}\geq& (1-C_1M^{-1}-C_2M^{-4}-C_4M^{-\frac{k-l}{k-1}})\frac{\xi_1^2}{\k_1^2}\\
       \geq& (1-C_5M^{-\frac{k-l}{k-1}})\frac{\xi_1^2}{\k_1^2}.
     \end{aligned}
 \end{equation}
By choosing $M\geq C_5\epsilon^{-\frac{k-1}{k-l}}$, we obtain the inequality \eqref{lu,iq}. Thus using Proposition \ref{l,N}, we may choose $N(\epsilon)=C_5\epsilon^{-\frac{k-1}{k-l}}$ and complete the proof by assuming $\k_1\geq N^{2^{k-1}}$.
\end{proof}

Now we state a crucial calculation by Huisken-Sinestrari \cite[Theorem 2.5]{HS} (see also Guan\cite{guannote}).
\begin{lem}\label{HS}
Suppose $\{\k_i\}\in \Gamma_2$. Then 
\begin{align}\label{q2}
    -\partial^2_{\xi}q_2=\sum_{i=1}^n\frac{\f\xi_i-\frac{\k_i}{\s_1}\s_1(\xi)\r^2}{\s_1}.    
\end{align}
 Suppose $\{\k_i\}\in \Gamma_{k+1}$. Let $[\xi]_i$ and $[\k]_i$ be vectors by setting the $i$th component of $\xi$ and $\k$ equal to zero respectively. Then
   \begin{align}\label{qq2}
       -\partial^2_{\xi}q_{k+1}\geq -\sum_{i=1}^n\frac{\k_i^2\partial^2_{[\xi]_i}q_k([\k]_i)}{(k+1)(q_{k;i}+\k_i)^2}.
   \end{align}
   where $q_{k+1}=\frac{\s_{k+1}}{\s_k}$ and $q_{k;i}=\frac{\s_k(\k|i)}{\s_{k-1}(\k|i)}$.
\end{lem}
We note that  by \eqref{q2}, we have
\begin{align}
    -\partial^2_{\xi}q_2\geq \frac{\left|[\xi]^{\perp}\right|^2}{\s_1}
\end{align}
where we define by $[\xi]^{\perp}$ the orthogonal part of $\xi$ with respect to the vector $(\k_1,\dots,\k_n)$. More generally, we have the following lemmas that are crucial to the proof of the Theorem \ref{main}.

\begin{lem}\label{HS1}
Suppose $\{\k_i\}\in \Gamma_{l+2}$. Then
\begin{align}\label{q2corollary1-1}
   -\partial^2_\xi q_{2;i_1,\cdots,i_\ell}\geq \frac{|[\xi]_{i_1,\cdots,i_\ell}^\perp|^2}{\s_1(\k|i_1,\cdots,i_\ell)}
\end{align}
 where $[\xi]_{i_1,\cdots,i_\ell}$ denotes the vector that is obtained by setting the $i_1$th,$\cdots,$ $i_\ell$th components of $\xi$ equal to zero and $[\xi]_{i_1,\cdots,i_\ell}^\perp$ denotes orthogonal part of $[\xi]_{i_1,\cdots,i_\ell}$ with respect to $[\k]_{i_1,\cdots,i_\ell}.$
Moreover, suppose $\{\k_i\}\in \Gamma_{k+2}$. Then
\begin{align}\label{qq2corollary1}
      - \partial^2_{\xi}q_{k+1;i_1}\geq -\sum_{i=\{1,\cdots, n\}\setminus i_1}\frac{\k_i^2\partial^2_{[\xi]_{i_1,i}}q_k([\k]_{i_1,i})}{(k+1)(q_{k;i_1,i}+\k_i)^2}
   \end{align}   
\end{lem}

\begin{proof}
From \eqref{q2} in Lemma \ref{HS} it follows
\begin{align}\label{q2corollary1}
   -\partial^2_\xi q_{2;i_1,\cdots,i_\ell}=\sum_{i=\{1,\cdots,n\}\setminus\{i_1,\cdots,i_\ell\}}    \frac{\f\xi_i-\frac{\k_i}{\s_1(\k|i_1,\cdots,i_\ell)}\s_1(\xi|i_1,\cdots,i_\ell)\r^2}{\s_1(\k|i_1,\cdots,i_\ell)}
\end{align}
Then \eqref{q2corollary1-1} holds by decomposing $[\xi]_{i_1,\cdots,i_\ell}$ with respect to $[\k]_{i_1,\cdots,i_\ell}$, i.e., write $[\xi]_{i_1,\cdots,i_\ell}=c[\k]_{i_1,\cdots,i_\ell}+[\xi]^\perp_{i_1,\cdots,i_\ell}$.

\eqref{qq2corollary1} follows from \eqref{qq2} .
\end{proof}

\begin{lem}\label{lem,HG}
Assume $(\k_i)\in \Gamma_k$, then 
    \begin{align}\label{sigma_k-1hessianqk}
       -\s_{k-1}\partial^2_\lambda q_k\geq C(n,k)\sum_{j=1}^{k-1}\k_1\cdots\hat{\k}_j\cdots\k_{k-1}\left|[\lambda]^\perp_{1,\cdots,\hat{j},\cdots,k-1}\right|^2.
   \end{align}
   Furthermore, denote $\lambda^l=(\lambda_{l+1},\dots\lambda_{n})$ to be the $(n-l)$-vector. Suppose that $|\k_n|>\d_0$ for some positive integer $\d_0$ and $\k_p\leq N_2$ for some $l+1\leq p\leq k-1$ (this is a condition we will assume in the main proof). We assume that $\lambda^l\perp (\k_{l+1},\dots,\k_n) $ for some $l\in (1,\cdots,k-1)$, then
   \begin{align}\label{qk2}
       -\s_{k-1}\partial^2_{\lambda}q_k\geq 
       \frac{C\d_0}{N_2^2}\k_1\k_2\cdots\k_{k-1}|\lambda^l|^2.
   \end{align}
   where $C$ depends on $n,k,l, N$.
   
\end{lem}

\begin{proof}

We note that 
\begin{align}
   q_{k;i}+\k_i=\frac{\s_{k+1}^{ii}+\k_i\s_k^{ii}}{\s_k^{ii}}=\frac{\sigma_k}{\s_{k-1}(\k|i)}.
\end{align}
By \eqref{qq2}, \eqref{qq2corollary1} and \eqref{q2corollary1-1},  we have
\begin{equation}\label{qk1}
  \begin{aligned}
    -\s_{k-1}\partial^2_{\lambda}q_k\geq& -\frac{\k_{i_1}^2\partial_{[\lambda]_{i_1}}^2q_{k-1([\k]_{i_1})\s_{k-2}(\k|i_1)}}{k\frac{\s_{k-1}}{\s_{k-2}(\k|i_1)}}\\ \geq& -\frac{\k_{i_1}^2\k_{i_2}^2\partial_{[\lambda]_{i_1,i_2}}^2 q_{k-2}([\k]_{i_1,i_2})\s_{k-3}(\k|i_1,i_2)}{k(k-1)\frac{\s_{k-1}}{\s_{k-2}(\k|i_1)}\frac{\s_{k-2}(\k|i_1)}{\s_{k-3}(\k|i_1,i_2)}}\\
    \geq &\cdots\\
    \geq &-\frac{2\k_{i_1}^2\k_{i_2}^2\cdots\k_{i_{k-2}}^2\partial_{[\lambda]_{i_1,i_2,\dots,i_{k-2}}}^2 q_2([\k]_{i_1,i_2,\dots,i_{k-2})\s_1(\k|i_1,i_2,\cdots,i_{k-2})}}{k!\frac{\s_{k-1}}{\s_1(\k|i_1,i_2,\cdots,i_{k-2})}}\\
    \geq &\frac{2\k_{i_1}^2\k_{i_2}^2\cdots\k_{i_{k-2}}^2|[\lambda]_{i_1,i_2,\cdots,i_{k-2}}^\perp|^2}{k!\frac{\s_{k-1}}{\s_1(\k|i_1,i_2,\cdots,i_{k-2})}}\\
\end{aligned}  
\end{equation}
where we use \eqref{q2corollary1-1} in the last inequality.
Taking $(i_1,i_2,\cdots,i_{k-2})$ from $(1,2,\cdots,k-1)$ and denote the remaining index by $j$,
by using the facts that $\s_{k-1}\leq C(n,k)\k_1\cdots\k_{k-1}$ and $\s_1(\k|i_1,i_2,\cdots,i_{k-2})\geq \k_{j}$, we obtain
\begin{align}
     -\s_{k-1}\partial^2_{\lambda}q_k\geq \sum_{(i_1,i_2,\cdots,i_{k-2})\in (1,2,\cdots,k-1)}C(n,k)\k_{i_1}\cdots\k_{i_{k-2}}|[\lambda]^\perp_{i_1,i_2,\cdots,i_{k-2}}|^2.
\end{align}
and complete the proof of \eqref{sigma_k-1hessianqk}.

Further, if $(\lambda_{l+1},\dots\lambda_{n})\perp (\k_{l+1},\dots,\k_n) $ for some $l\in (1,\cdots,k-2)$. We obtain by \eqref{sigma_k-1hessianqk} that,
 \begin{align*}
       -\s_{k-1}\partial^2_\lambda q_k\geq C\sum_{j=l+1}^{k-1}\k_1\cdots\hat{\k}_j\cdots\k_{k-1}\left|[\lambda]^\perp_{1,\cdots,\hat{j},\cdots,k-1}\right|^2\geq \frac{C(n,k)}{N_2}\k_1\cdots\k_{k-1}\sum_{j=l+1}^{k-1}\left|[\lambda]^\perp_{1,\cdots,\hat{j},\cdots,k-1}\right|^2.
   \end{align*}

Next, we will prove by contradiction. Denote some vectors in $\R^{n-l}$ by \[\lambda=(\lambda_{l+1},\dots,\lambda_n)  ,\a=(\k_{l+1},\dots,\k_n),\a_k=(0,\dots,0,\k_k,\dots,\k_n),\lambda_k=(0,\dots,0,\lambda_k,\dots,\lambda_n),\]
\[\lambda_p=(0,\dots,0,\lambda_p,0,\dots,0,\lambda_k,\cdots,\lambda_n),\quad\xi_p=(0,\dots,0,\xi_p,0,\dots,0,\xi_k,\cdots,\xi_n),\]
and \[\a_p=(0,\dots,0,\k_p,0,\dots,0,\k_k,\dots,\k_n)\]
where $\xi_p\perp\a_p.$
We assume that $\lambda_p=\xi_p+a_p\a_p$. 

Suppose there exists $0
<c_0\ll 1$, such that $|\xi_p|<c_0|\lambda|$ for any $l+1\leq p\leq k-1$, i.e.,
\[\lambda_k-a_p\a_k\leq c_0|\lambda|.\]
Then for any $p\neq q\in (l,k)$
\begin{align}
    |(a_p-a_q)\a_k|\leq|\lambda_k-a_p\a_k|+|\lambda_k-a_q\a_k|\leq 2c_0|\lambda|.
\end{align}
Then \[|a_p-a_q|\leq \frac{2c_0|\lambda|}{|\a_k|}.\]
and \[|\lambda_q-a_p\a_q|\leq|\lambda_q-a_q\a_q|+|(a_p-a_q)\a_q|\leq c_0|\lambda|+\frac{2c_0|\lambda||\a_q|}{|\a_k|}.\]
Thus 
\begin{equation}\label{c0}
 \begin{aligned}
    |\lambda-a_p\a|\leq&\sum_{q=l+1}^{k-1}|\lambda_q-a_p\a_q|\\
    \leq &(k-1)c_0|\lambda|+\frac{2k(n-k+2)N_2c_0|\lambda|}{\d_0}\\
    < &\frac{3nkN_2c_0|\lambda|}{\d_0}
\end{aligned}   
\end{equation}
where we use Lemma \ref{lemn} in the lat inequality.
Since $\lambda\perp\a$, we have 
\begin{align}
    |\lambda-a_p\a|^2\geq |\lambda|^2.
\end{align}
Thus we complete the proof of \eqref{qk2} by only choosing $c_0=\frac{\d_0}{3nkN_2}$.
\end{proof}

\section{proof of Theorem \ref{main}}\label{sec:3}

For simplicity, in this section we only show the dependence of positive constants $\epsilon<1$, $C$, $c$, $C_i$ and $c_i$ for any $i\in \N$ on other factors, rather than on $n,k,\s_k$ or $A$, unless further specified in the following part. Keep in mind that $F=\sigma_k$ is a constant. 

Consider the auxillary function $Q=\ln\k_1-L\ln\nu^{n+1}$ with positive constant $L$ to be considered below. Since the multiplicity of $\k_1$ may be strictly larger than 1. We apply the perturbation argument (see \cite{chu}) here. Assume $Q$ achieves the maximum at the interior point $p$, we choose an orthonormal basis  $e_1(p),\dots,e_n(p)$ such that $h_{ij}=\k_i\delta_{ij}$ and 
    \begin{align*}
        \k_1\geq\k_2\geq\cdots\geq \k_n.
    \end{align*}
    Let $\gamma(s)$ be any geodesic starting at $p$ satisfying $\gamma'(0)=e_i(p)$. We construct an orthonormal frame $e_1(s), e_2(s),\dots,e_n(s)$ near $p$ by extending $e_1(p), e_2(p),\dots,e_n(p)$ parallel along $\gamma(s)$. We construct a unit vector field $v(s)$ near $p$ such that $v(0)=e_1(p)$ and $v'(0)=\sum_{p\neq 1}\frac{h_{1pi}}{\k_1+1-\k_p}e_p$. 
Consider a new function 
\begin{align}
    \tilde{Q}=\ln \f h\f v(s),v(s)\r+g\f v(s),e_1(s)\r^2-1\r-L\ln \nu^{n+1}.
\end{align}
 We note that $\tilde{Q}$ also attains its local maximum at $s=0$. Then 
    \begin{align}\label{gra}
        \frac{h_{11i}}{\k_1}=L\frac{\nu^{n+1}_i}{\nu^{n+1}},
    \end{align}
    \begin{align}\label{hess}
        \frac{h_{11ii}}{\k_1}+\sum_{p\neq 1}\frac{2h_{1pi}^2}{\k_1(\k_1+1-\k_p)}-\frac{h_{11i}^2}{\k_1^2}-\frac{L\nu^{n+1}_{ii}}{\nu^{n+1}}+L\frac{(\nu^{n+1
}_i)^2}{(\nu^{n+1})^2}\leq 0.
    \end{align}
By the Ricci identity and Codazzi identity, we have 
\begin{align}\label{ric}
    h_{11ii}=h_{ii11}-\k_i^2\k_1+\k_1^2\k_i+\k_i-\k_1\delta_{ii}.
\end{align}
Multiply \eqref{hess} by $F^{ii}$. Then 
\begin{align}\label{hess2}
  \frac{F^{ii}h_{11ii}}{\k_1}+2\sum_{p\neq 1}\frac{F^{ii}h_{1pi}^2}{\k_1(\k_1+1-\k_p)}-\frac{F^{ii}h_{11i}^2}{\k_1^2}-L\frac{F^{ii}\nu^{n+1}_{ii}}{\nu^{n+1}}+L\frac{F^{ii}(\nu^{n+1
}_i)^2}{(\nu^{n+1})^2}\leq 0 .  
\end{align}
Recall that 
\begin{align}\label{sk2}
    \s_k=C_n^k\s.
\end{align}
Differentiate \eqref{sk2} twice and obtain
\begin{align}\label{skhess}
\dot{F}^{ij}h_{ij11}+\ddot{F}^{pq,rs}h_{pq1}h_{rs1}=0.
\end{align}
We use Lemma \ref{nu,l}, plug \eqref{ric}, \eqref{skhess} into \eqref{hess2} and obtain
\begin{equation}\label{hess3}
    \begin{aligned}
  &-\frac{\ddot{F}^{pq,rs}h_{pq1}h_{rs1}}{\k_1}+2\sum_{p\neq 1}\frac{F^{ii}h_{1pi}^2}{\k_1(\k_1+1-\k_p)}-\dot{F}^{ii}\k_i^2+kF(\k_1+\frac{1}{\k_1})-\dot{F}^{ii}-\frac{F^{ii}h_{11i}^2}{\k_1^2}\\
&+L\f\dot{F}^{ii}\k_i^2+\dot{F}^{ii}\r-2L\frac{\dot{F}^{ii}u_i\nu^{n+1}_i}{u\nu^{n+1}}-kLF(\frac{1}{\nu^{n+1}}+\nu^{n+1})+L\frac{F^{ii}(\nu^{n+1
}_i)^2}{(\nu^{n+1})^2}\leq 0.
\end{aligned}
\end{equation}
By direct computation, we obtain
\begin{align}\label{mp}
    2\sum_{p\neq 1}\frac{F^{ii}h_{1pi}^2}{\k_1(\k_1+1-\k_p)}\geq 2\sum_{p\neq 1}\frac{F^{pp}h_{pp1}^2}{\k_1(\k_1+1-\k_p)}+2\sum_{i\neq 1}\frac{F^{11}h_{11i}^2}{\k_1(\k_1+1-\k_i)}.
\end{align}
Due to \cite{Ben}, we have
\begin{align}\label{org}
-\ddot{F}^{pq,rs}h_{pq1}h_{rs1}=-\ddot{F}^{pp,qq}h_{pp1}h_{qq1} +2\sum_{p>2}\ddot{F}^{pp,qq}h_{pq1}^2\geq -\ddot{F}^{pp,qq}h_{pp1}h_{qq1} + 2\sum_{p\neq 1}\frac{\dot{F}^{pp}-\dot{F}^{11}}{\k_1-\k_p} h_{11p}^2.
\end{align}
Plugging \eqref{mp} and \eqref{org} into \eqref{hess3}, we obtain
\begin{equation}\label{hess4}
    \begin{aligned}
     0\geq&-\frac{\ddot{F}^{pp,qq}h_{pp1}h_{qq1}}{\k_1}+2\sum_{p\neq 1}\frac{F^{pp}h_{pp1}^2}{\k_1(\k_1+1-\k_p)}+2\sum_{i\neq 1}\frac{\dot{F}^{ii}h_{11i}^2}{\k_1(\k_1+1-\k_i)} -\sum_{i=1}^{n}\frac{F^{ii}h_{11i}^2}{\k_1^2}\\
&+(L-1)\f\dot{F}^{ii}\k_i^2+\dot{F}^{ii}\r-2L\frac{\dot{F}^{ii}u_i\nu^{n+1}_i}{u\nu^{n+1}}-CL+L\frac{F^{ii}(\nu^{n+1
}_i)^2}{(\nu^{n+1})^2}
    \end{aligned}
\end{equation}
where we use Lemma \ref{bnu} in the next to last term of \eqref{hess3}. Recall 
\eqref{nu2}
\begin{align}
    \nabla_i\nu^{n+1}=(\nu^{n+1}-\k_i)\frac{u_i}{u}. 
\end{align}
We note that 
\begin{align}
-2L\frac{\dot{F}^{11}u_1\nu^{n+1}_1}{u\nu^{n+1}} =-2L\frac{\dot{F}^{11}u_1^2(\nu^{n+1}-\k_1)}{u^2\nu^{n+1}}>0
\end{align}
where we assume that $\k_1>1$.
Besides, since $\k_i\geq -A$, we have
\begin{align}\label{h11i}
   2\sum_{i\neq 1}\frac{\dot{F}^{ii}h_{11i}^2}{\k_1(\k_1+1-\k_i)} -\sum_{i\neq 1}^{n}\frac{F^{ii}h_{11i}^2}{\k_1^2} =\sum_{i\neq 1}\frac{\k_1-1+\k_i}{\k_1+1-\k_i}\frac{\dot{F}^{ii}h_{11i}^2}{\k_1^2}\geq \frac{1}{2} \frac{\dot{F}^{ii}h_{11i}^2}{\k_1^2}.
\end{align}
Here we assmue $\k_1\geq 3A+3$. Thus by using the Cauchy-Schwartz inequality and the critical equation \eqref{gra}, we get
\begin{equation}\label{hess5}
    \begin{aligned}
       0\geq &-\frac{\ddot{F}^{pp,qq}h_{pp1}h_{qq1}}{\k_1}+2\sum_{p\neq 1}\frac{F^{pp}h_{pp1}^2}{\k_1(\k_1+1-\k_p)}+\sum_{i\geq 2}^n\frac{\dot{F}^{ii}h_{11i}^2}{2\k_1^2}-\frac{F^{11}h_{111}^2}{\k_1^2}\\
       &+(L-1)\f\dot{F}^{ii}\k_i^2+\dot{F}^{ii}\r-\sum_{i\geq 2}^n\frac{L^2\dot{F}^{ii}(\nu^{n+1}_i)^2}{2(\nu^{n+1})^2}-\sum_{i\geq 2}^n 2\frac{\dot{F}^{ii}u_i^2}{u^2}-LC\\
       \geq &-\frac{\ddot{F}^{pp,qq}h_{pp1}h_{qq1}}{\k_1}+2\sum_{p\neq 1}\frac{F^{pp}h_{pp1}^2}{\k_1(\k_1+1-\k_p)}-\frac{F^{11}h_{111}^2}{\k_1^2}\\
       &+(L-1)\f\dot{F}^{ii}\k_i^2+\dot{F}^{ii}\r-2\sum_{i\geq  2}\dot{F}^{ii}-LC.
    \end{aligned}
\end{equation}
Similar to \eqref{h11i}, we have
\begin{equation}\label{hii1}
 \begin{aligned}
   2\sum_{i\neq 1}\frac{\dot{F}^{ii}h_{ii1}^2}{\k_1(\k_1+1-\k_i)}=&\sum_{i\neq 1}\f \frac{2}{\k_1^2}+\frac{2(\k_i-1)}{\k_1^2(\k_1+1-\k_i)}\r\dot{F}^{ii}h_{ii1}^2\\
   \geq &2\sum_{i\neq 1}^{n}\frac{F^{ii}h_{ii1}^2}{\k_1^2} -\sum_{i\neq 1}\frac{A+1}{\k_1^2(\k_1+A+1)}\dot{F}^{ii}h_{ii1}^2\\
   \geq &\f 2-\frac{C(A)}{\k_1}\r\sum_{i\neq 1}^{n}\frac{F^{ii}h_{ii1}^2}{\k_1^2}.
\end{aligned}   
\end{equation}

Choose $L=4$ in \eqref{hess5}. We obtain
\begin{align}\label{hess61}
    0\geq -\frac{\ddot{F}^{pp,qq}h_{pp1}h_{qq1}}{\k_1}+\f 2-\frac{C(A)}{\k_1}\r\sum_{p\neq 1}\frac{F^{pp}h_{pp1}^2}{\k_1^2}-\frac{F^{11}h_{111}^2}{\k_1^2}+\dot{F}^{ii}\k_i^2+\dot{F}^{ii}-C.
\end{align}
\eqref{gra} and \eqref{nu2} implies that
\begin{align}\label{h111}
     \frac{F^{11}h_{111}^2}{\k_1^2}=\frac{16\dot{F}^{11}(\nu^{n+1}-\k_1)^2}{(\nu^{n+1})^2}\frac{u_1^2}{u^2}\leq\frac{16\dot{F}^{11}\k_1^2}{a_1^2}.
\end{align}
For the quantitative relationship between curvatures, we divide it into the following cases

\textbf{\emph{Case 1} 
\begin{align}
    (n-k+1)\s_{k-1}\geq \frac{16}{a_1^2}\f\s_1\s_k-(k+1)\s_{k+1}\r
\end{align}}
This is equivalent to
\begin{align}\label{case1}
    \sum_{i=1}^n\dot{F}^{ii}\geq \frac{16}{a_1^2}\sum_{i=1}^n\dot{F}^{ii}\k_i^2.
\end{align}
Differeniating \eqref{sk2} we have
\begin{align}\label{gra2}
\sum_{i=1}^n\dot{F}^{ii}h_{ii1}=0.
    \end{align}
    Thus, combining \eqref{gra2} and \eqref{xi1}, we have 
    \begin{align}
     -\frac{\ddot{F}^{pp,qq}h_{pp1}h_{qq1}}{\k_1}= -\frac{F\ddot{(\log F)}^{pp,qq}h_{pp1}h_{qq1}}{\k_1}\geq 0
     \end{align}
Plugging \eqref{h111} and \eqref{case1} into \eqref{hess61}, we obtain
\begin{align*}
     0\geq \dot{F}^{11}\k_1^2-C 
\end{align*}
Due to Lemma \ref{fk1}, we obtain
\begin{align*}
\k_1\leq C_1 ,
\end{align*}
which completes the proof of this case.

\vskip.2cm

\textbf{\emph{Case 2} 
\begin{align}\label{sk-1}
 (n-k+1)\s_{k-1}<\frac{16}{a_1^2}\f\s_1\s_k-(k+1)\s_{k+1}\r.   
\end{align}}

\vskip.2cm

\textbf{\emph{Case 2(a)} There exists a small $\d>0$, such that \[\k_n>-\d,\quad \text{i.e.}, \quad|\k_n|<\d.\] }

Following the proof of Lemma \ref{upperboundforkappak}, we obtain by \eqref{kk,1}
 \begin{align}\label{kappakNbound}
    \k_k<C\k_1^{-\frac{1}{k-1}}+C\d.
\end{align}
 Due to Lemma \ref{lemn}, Lemma \ref{k1} and Cauchy-Schwarz inequality,
\begin{align}
    |\s_{k+1}|\leq C(n,k)\k_1\cdots\k_{k-1}\k_k^2\leq C(n,k)\s_{k-1}\k_k^2\leq C_2\f\k_1^{-\frac{2}{k-1}}+\d^2\r\s_{k-1}
\end{align}
If we choose $\d\leq \sqrt{\frac{(n-k+1)a_1^2}{48(k+1)C_2}}$ and $\k_1\geq \f\frac{(n-k+1)a_1^2}{48(k+1)C_2}\r^{-\frac{k-1}{2}}$ such that
\begin{equation}\label{d0,k1}
  \frac{16(k+1)}{a_1^2}|\s_{k+1}|\leq\frac{16(k+1)C_2}{a_1^2}\f\d^2+\k_1^{-\frac{2}{k-1}}\r\s_{k-1}\leq \frac{2(n-k+1)\s_{k-1}}{3}.  
\end{equation}
Then inserting \eqref{d0,k1} into \eqref{sk-1}, we have 
\begin{align}
    \k_1\k_2\cdots\k_{k-1}<\s_{k-1}\leq C\s_1<nC\k_1.
\end{align}
Thus $\k_2\cdots\k_{k-1}\leq C$. By \eqref{kappakNbound}, we obtain
\begin{align}
    |\s_k(\k|1)|\leq C(n,k)\k_2\cdots\k_{k-1}\k_k^2\leq C\f\k_1^{-\frac{2}{k-1}}+\d^2\r.
\end{align}
Since $\sigma_k=\kappa_1 \dot{F}^{11}+\s_k(\kappa|1)$, then
\begin{align}
\k_1\dot{F}^{11}-C_3\s_k\f\k_1^{-\frac{2}{k-1}}+\d^2\r\leq \s_k,
\end{align}

Equivalently,
\begin{align}\label{case2a}
    \dot{F}^{11}\leq \f1+C_3\f\k_1^{-\frac{2}{k-1}}+\d^2\r\r\frac{\s_k}{\k_1}
\end{align}
holds for some positive constant $C_3$ independent of $\d$. Inserting \eqref{case2a} into \eqref{hess61}, we have by using Lemma \ref{lu}
\begin{equation}\label{d0,k1,2}
 \begin{aligned}
    0\geq& (1-\epsilon)\frac{\s_k h_{111}^2}{\k_1^3}-\frac{\dot{F}^{11}h_{111}^2}{\k_1^2}+\dot{F}^{ii}\k_i^2+\dot{F}^{ii}-C\\
    \geq& \f\frac{(1-\epsilon)}{1+C_3\f\k_1^{-\frac{2}{k-1}}+\d^2\r}-1\r\frac{\dot{F}^{11}h_{111}^2}{\k_1^2}+\dot{F}^{11}\k_1^2-C\\
    \geq& \f1-\frac{16C_3\k_1^{-\frac{2}{k-1}}}{a_1^2}-\frac{16C_3\d^2}{a_1^2}-\frac{16\epsilon}{a_1^2}\r\dot{F}^{11}\k_1^2-C
\end{aligned}   
\end{equation}
where we use \eqref{h111} in the last inequality.
Then for the same reason in case 1, we have $\k_1<C_2$ by choosing $\k_1\geq \f\frac{a_1^2}{64C_3}\r^{-\frac{k-1}{2}}$, $\d\leq \sqrt{\frac{a_1^2}{64C_3}}$ and $\epsilon\leq\frac{a_1^2}{64}$. Using Lemma \ref{lu}, we obtain that there exists $N(\epsilon)>0$ only depending on $\epsilon$ such that if $\k_1\geq N^{2^{k-1}}$, then the first inequality in \eqref{d0,k1,2} holds.

 Denote by $\d_0=\min\lbrace \frac{a_1^2}{64C_1}, \sqrt{\frac{(n-k+1)a_1^2}{32(k+1)C_3}}\rbrace$ only depending on $n,k,\s_k$ and $a_1$. We conclude that combining \eqref{d0,k1} and \eqref{d0,k1,2}, for any $\d\leq \d_0$, if $\k_n\geq-\d$, then we have $\k_1\leq \max\lbrace\f\frac{(n-k+1)a_1^2}{48(k+1)C_2}\r^{-\frac{k-1}{2}}, \f\frac{a_1^2}{64C_3}\r^{-\frac{k-1}{2}}, C_2\rbrace$ independent of $\d_0$.

\vskip.3cm

\textbf{\emph{Case 2(b)}  $$-A<\k_n<-\d_0,\ \text{i.e.,}\ \d_0<|\k_n|<A.$$ }
\vskip.1cm

    We first prepare some estimates under the case $2(b)$ that will be beneficial to the  subsequent proof. As the positive constant $\d_0$ in this case only depends on $n,k,\s_k$, unless explicitly stated we ignore the dependence of constant $C$, $c$, $C_i$ and $c_i$ for any $i\in \N$ on $\d_0$ in the following discussions.

\begin{prop}\label{betasigmak-1kappa1}
   Denote by $\b=(0,\dot{F}^{22},\dots,\dot{F}^{nn})$. Then
\begin{align}\label{betalowerbound1}
|\b|\geq c_1\sigma_{k-1}
\end{align} and 
\begin{align}\label{sigmak-1lowerbound1}
\sigma_{k-1}\geq C(n,k,\d_0)\k_1.
\end{align}
 
\end{prop}

\begin{proof}
    By using the Cauchy-Schwartz inequality we obtain
\begin{equation}\begin{aligned}\label{betaupperbound-1}
|\b|&\geq \frac{\sum_{p=2}^n\dot{F}^{pp}}{\sqrt{n-1}}\geq \frac{(n-k+1)\sqrt{n-1}}{n}\s_{k-1}\\ &\geq \frac{(n-k+1)\sqrt{n-1}}{n}\k_1\cdots\k_{k-1}\geq \frac{(n-k+1)\sqrt{n-1}}{n}\k_1 \k_{k}^{k-2}\\&\geq \frac{(n-k+1)\sqrt{n-1}}{n}\f\frac{\d_0}{(n-k)}\r^{k-2}\k_1
\end{aligned}
\end{equation}
where we use Lemma \ref{lem'} in the second inequality, Lemma \ref{k1} in the third inequality and Lemma \ref{lemn} in the last inequality.
\end{proof}

\begin{prop}\label{f11est}
\begin{align} 
   \s_{k-1}(\k|1)\leq C\frac{\s_{k-1}}{\k_1^2}.
\end{align}
\end{prop}

\begin{proof}
     Recall that 
 \begin{align}\label{sk}
 \s_{k}=\k_1\s_{k-1}(\k|1)+\s_{k}(\k|1)    
 \end{align}
Since by Lemma \ref{lemn} and Lemma \ref{upperboundforkappak},
\begin{align}
    |\s_{k}(\k|1)|\leq C(n,k)\k_2\cdots\k_{k-1}\k_k^2\leq C(n,k,A)\frac{\s_{k-1}}{\k_1}.
\end{align}
Then we have 
\begin{align}
  |\s_{k-1}(\k|1)|=\left|\frac{\s_{k}}{\k_1}-\frac{\s_{k}(\k|1)}{\k_1}\right|\leq  \frac{\s_k}{\k_1}+ C(n,k,A)\frac{\s_{k-1}}{\k_1^2}.
\end{align}
 We complete the proof of Proposition \ref{f11est} by using \eqref{sigmak-1lowerbound1}.
\end{proof}

 \begin{prop}\label{k1f11,f}
   If $\k_1>\frac{(n-k)^{k-1}\s_k}{\d_0^{k-1}}$, then
\begin{align*} 
   \s_k(\k|1)\leq 0.
\end{align*}
In other words,
\begin{align*}
    \k_1\dot{F}^{11}\geq\s_k.
\end{align*}
\end{prop}
\begin{proof}
We prove it by contradiction. Assume that
\begin{align}\label{sk1,c}
   \s_k(\k|1)>0. 
\end{align}

Then from Lemma \ref{k1}, we have 
\[F^{11}>\k_2\k_3\cdots\k_k\] and 
\begin{align}\label{k1f11,c}
\k_1\dot{F}^{11}>\k_1\k_2\cdots\k_k\geq \left(\frac{\d_0}{n-k}\right)^{k-1}\k_1
\end{align}
Here we use Lemma \ref{lemn} and the assumption $|\k_n|>\d_0$ in this case.
Inserting \eqref{sk1,c} and \eqref{k1f11,c} into \eqref{sk}, we derive a contradiction by assuming $\k_1>\frac{(n-k)^{k-1}\s_k}{\d_0^{k-1}}.$
\end{proof}

\textbf{Main claim:} when $\k_1$ is large enough, the first three terms in \eqref{hess61} satisfy
\begin{align}\label{claim}
-\frac{\ddot{F}^{pp,qq}h_{pp1}h_{qq1}}{\k_1}+\f 2-\frac{C(A)}{\k_1}\r\sum_{p\neq 1}\frac{\dot{F}^{pp}h_{pp1}^2}{\k_1^2}- \frac{\dot{F}^{11}h_{111}^2}{\k_1^2}\geq 0.
\end{align}
Then \eqref{hess61} implies that there exists some positive constant $C_{3}$ depending only on $n,k,\s_k$ such that $\k_1\leq C_{3}$. Our main goal is to prove this claim.

Applying Lemma \ref{l,N}, in the subsequent context we suppose that there exists an integer $1\leq \ell\leq k-1$ for which $\k_{\ell}\geq N_2^2$ and $\k_{\ell+1}\leq N_2$.
\begin{prop}
 Suppose that $1\leq \ell\leq k-1$. Denote $\xi=(h_{111},h_{221},\dots,h_{nn1})$ and recall $\k=(\k_1,\k_2,\dots,\k_n)$ and  $\b=(0,\dot{F}^{22},\dots,\dot{F}^{nn})$. If $\k_1$ is large, $\xi$ can be uniquely written as 
\begin{align}\label{xi}
    \xi=(1+a)\frac{h_{111}}{\k_1}\k-ah_{111}e_1+\frac{b\b}{|\b|^2}+(0,\lambda_2,\dots,\lambda_{l+1},\dots,\lambda_n)
\end{align}
where $e_1=(1,0,\dots,0)$, 
\begin{align}\label{perp}
(\lambda_{l+1},\dots\lambda_{n})\perp (\k_{l+1},\dots,\k_n)    
\end{align}
and 
\begin{align}\label{lambda}
\sum_{\a=2}^n\dot{F}^{\a\a}\lambda_{\a}=0.
\end{align}  
Besides,
\begin{align}\label{b}
    b=\f a-(1+a)\frac{kF}{\k_1\dot{F}^{11}}\r\dot{F}^{11}h_{111}.
\end{align}
\end{prop}
\begin{proof}
    We multiply \eqref{xi} by $\lbrace\dot{F}^{ii}\rbrace$ and obtain using \eqref{gra2}
\begin{align}\label{ab1}
    (1+a)\frac{kF}{\k_1}h_{111}-a\dot{F}^{11}h_{111}+b=0.
\end{align}
Denote by $\k^l=(\k_{l+1},\dots,\k_n)$. We multiply \eqref{xi} by $\k^l$ and obtain using \eqref{perp}
\begin{align}\label{ab2}
    (1+a)\frac{h_{111}}{\k_1}\sum_{p=l+1}^n\k_p^2+b\sum_{p=l+1}^n\frac{\dot{F}^{pp}\k_p}{|\b|^2}=\metric{\xi}{\k^l}.
\end{align}
We only need to verify that the following matrix is invertible,
$$
\f
\begin{array}{cc}
   kF-\k_1\dot{F}^{11} & 1 \\
    \sum_{p=l+1}^n\k_p^2 & \frac{\sum_{p=l+1}^n\dot{F}^{pp}\k_p}{|\b|^2}
\end{array}
\r
$$
i.e.
\begin{equation}
    \f kF-\k_1\dot{F}^{11}\r\sum_{p=l+1}^n\dot{F}^{pp}\k_p-|\b|^2 \sum_{p=l+1}^n\k_p^2\neq 0.
\end{equation}
Combining \eqref{sigmak-1lowerbound1} and Proposition \ref{f11est}, we have 
\begin{equation}
   |kF-\k_1\dot{F}^{11}|\leq C\frac{\s_{k-1}}{\k_1}.
\end{equation}
Besides, by Lemma \ref{f11k12}
\begin{align}\label{det,k1}
|\dot{F}^{pp}\k_p|\leq|\s_k|+|\s_k(\k|p)|\leq C\s_{k-1}.
\end{align}

Due to \eqref{betalowerbound1} and $|\k_n|\geq \d_0$, we have 
\begin{align}
    \left|\f kF-\k_1\dot{F}^{11}\r\sum_{p=l+1}^n\dot{F}^{pp}\k_p\right|\leq C_1\frac{\s_{k-1}^2}{\k_1}<|\b|^2 \sum_{p=l+1}^n\k_p^2.
\end{align}

by choosing $\k_1>\frac{C_1}{c_1^2\d_0^2}$, where $c_1$ is derived in Proposition \ref{betasigmak-1kappa1}.
Further, by \eqref{ab1} we have 
\begin{align}
    b=\f a-(1+a)\frac{kF}{\k_1\dot{F}^{11}}\r\dot{F}^{11}h_{111}.
\end{align}
\end{proof}

We replace $h_{pp1}$ by $\xi_p$ in \eqref{claim} and obtain by using \eqref{xi}
\begin{equation}\label{hess6}
  \begin{aligned}
      -\ddot{F}^{pp,qq}\xi_p\xi_q+\f 2-\frac{C(A)}{\k_1}\r\sum_{p\neq 1}\frac{F^{pp}\xi_p^2}{\k_1}\geq & \underbrace{(1+a)^2\frac{h_{111}^2}{\k_1^2}\ddot{F}^{pp,qq}\k_p\k_q-2\frac{h_{111}}{\k_1}\ddot{F}^{pp,qq}\k_p\xi_q}_{I}\\
&\underbrace{+2ah_{111}\frac{b}{|\b|^2}\sum_{p=2}^n\ddot{F}^{11pp}\dot{F}^{pp}}_{II}+\underbrace{2ah_{111}\sum_{\gamma=2}^n\ddot{F}^{11\gamma\gamma}\lambda_{\gamma}}_{III}\\
&\underbrace{-\frac{b^2}{|\b|^4}\sum_{p,q\geq2}\dot{F}^{pp}\ddot{F}^{ppqq}\dot{F}^{qq}}_{IV}\underbrace{-2\frac{b}{|\b|^2}\sum_{p,\gamma\geq2}\dot{F}^{pp}\ddot{F}^{pp\gamma\gamma}\lambda_{\gamma}}_{V}\\
&\underbrace{-\sum_{\gamma,\g\geq2}\lambda_{\gamma}\ddot{F}^{\gamma\gamma\g\g}\lambda_{\g}}_{VI}\\
&\underbrace{+2\f1-\frac{C(A)}{\k_1}\r\sum_{p\neq 1}\frac{F^{pp}\xi_p^2}{\k_1}}_{VII}.
  \end{aligned}  
\end{equation}

Now we estimate terms in the right-hand side of \eqref{hess6}. 

For $I$. We have
 \begin{align}\label{a1}
     I=(1+a)^2\frac{h_{111}^2}{\k_1^2}\ddot{F}^{pp,qq}\k_p\k_q-2\frac{h_{111}}{\k_1}\ddot{F}^{pp,qq}\k_p\xi_q=(1+a)^2\frac{k(k-1)Fh_{111}^2}{\k_1^2}
 \end{align}
where we use \eqref{gra2}.

For $IV$.
Since $\k_{k-1}\geq \k_k\geq \frac{1}{n-k}|\k_n|\geq\frac{\d_0}{(n-k)}$, for any $p,q\geq 2$
\begin{align}\label{fppqqupperbound}
|\ddot{F}^{pp,qq}|\leq C(n,k)\k_1\cdots\k_{k-2}\leq C(n,k,\d_0)\sigma_{k-1}.
\end{align}
Furthermore, by Lemma \ref{fk1} and Proposition \ref{f11est}, employing \eqref{b} we have 
\begin{equation}
\begin{aligned}\label{d}
  IV= -\frac{b^2}{|\b|^4}\sum_{p,q\geq2}\dot{F}^{pp}\ddot{F}^{ppqq}\dot{F}^{qq}&\geq -C(n,k,\d_0)\f a-(1+a)\frac{kF}{\k_1\dot{F}^{11}}\r^2\frac{\f\dot{F}^{11}h_{111}\r^2}{|\beta|}\\&\geq- C(n,k,\d_0,A)(a^2+1)\frac{\dot{F}^{11}h_{111}^2}{\k_1^2}
\end{aligned}
\end{equation}

For the rest of terms, we divide the proof of the \textbf{Main claim} into three parts.

\subsection{Assumption 1:} we suppose that $l=1$, $3\leq k\leq n-2$ and there exists a positive constant $N_2>1$, such that $\k_2<N_2$.

For $II$ of \eqref{hess6}, due to 
\begin{align}\label{11pp}
\dot{F}^{pp}=\k_1\dot{F}^{11,pp}+\s_{k-1}(\k|1,p)
\end{align}
and \eqref{b},  we have
\begin{equation}\label{b11}
   \begin{aligned}
   2ah_{111}\frac{b}{|\b|^2}\sum_{p=2}^n\ddot{F}^{11pp}\dot{F}^{pp}=& 2ah_{111}\frac{b}{|\b|^2}\sum_{p=2}^n\f\frac{\dot{F}^{pp}}{\k_1}-\frac{\s_{k-1}(\k|1,p)}{\k_1}\r\dot{F}^{pp} \\
   =&2a\f a-(1+a)\frac{kF}{\k_1\dot{F}^{11}}\r\frac{\dot{F}^{11}h_{111}^2}{\k_1}\f1-\sum_{p=2}^n\frac{\s_{k-1}(\k|1,p)\dot{F}^{pp}}{|\b|^2}\r\\
   =& 2a\f a-(1+a)\frac{kF}{\k_1\dot{F}^{11}}\r\frac{\dot{F}^{11}h_{111}^2}{\k_1} \\
   &- 2a\f a-(1+a)\frac{kF}{\k_1\dot{F}^{11}}\r\frac{\dot{F}^{11}h_{111}^2}{\k_1} \sum_{p=2}^n\frac{\s_{k-1}(\k|1,p)\dot{F}^{pp}}{|\b|^2}
\end{aligned} 
\end{equation}
Note that $\dot{F}^{pp}\leq |\beta|$ and \begin{align}\label{upperboundforsigmak-1}
|\s_{k-1}(\k|1,p)|\leq C(n,k)\k_2\cdots\k_{k-1}\k_k\leq C(n,k,A)\frac{\s_{k-1}}{\k_1},
\end{align}
where we used the facts that $\k_k\leq C(A)$ in Lemma \ref{upperboundforkappak}. 
We now use Lemma \ref{fk1} and \eqref{betalowerbound1} to estimate the last term of \eqref{b11}. It is bounded by 
\begin{align}C(n,k,A)|a|(|a|+n|1+a|)\frac{\dot{F}^{11}h_{111}^2}{\k_1^2}\leq C'(n,k,A)(a^2+1)\frac{\dot{F}^{11}h_{111}^2}{\k_1^2}.\end{align}
Thus \eqref{b11} becomes
\begin{align}\label{a2}
    2ah_{111}\frac{b}{|\b|^2}\sum_{p=2}^n\ddot{F}^{11pp}\dot{F}^{pp}\geq 2a\f a-(1+a)\frac{kF}{\k_1\dot{F}^{11}}\r\frac{\dot{F}^{11}h_{111}^2}{\k_1}-C(a^2+1)\frac{\dot{F}^{11}h_{111}^2}{\k_1^2}
\end{align}

Besides, we note that 
\begin{align}\label{sk-1,1}
    \s_{k-1}\leq C\k_1\cdots\k_{k-1}\leq CN_2^{k-2}\k_1
\end{align}
and 
\begin{align}\label{f1k1,1}
    \dot{F}^{11}\k_1=\s_k-\s_k(\k|1)\leq C_1(n,k,A)N_2^{k-2}
\end{align}
since we assume $\k_2<N_2$ in this case. From \eqref{fppqqupperbound} and similar to the proof in the first inequality of \eqref{d}, $V$ of \eqref{hess6} becomes
\begin{equation}\label{b2}
 \begin{aligned}
-2\frac{b}{|\b|^2}\sum_{p,q\geq2}\dot{F}^{pp}\ddot{F}^{pp,qq}\lambda_{q}\geq& -C(n,k,\d_0)|b||\lambda|\\
\geq& -\frac{CC_1(n,k,A)N_2^{k-2}}{\epsilon}(a^2+1)\frac{\dot{F}^{11}h^2_{111}}{\k_1^2}-\frac{\epsilon\dot{F}^{11}\k_1^2|\lambda|^2}{C_1(n,k,A)N_2^{k-2}}\\
\geq& -\frac{CN_2^{k-2}}{\epsilon}(a^2+1)\frac{\dot{F}^{11}h^2_{111}}{\k_1^2}-\epsilon\k_1|\lambda|^2.
\end{aligned}   
\end{equation}
Here we use Cauchy-Schwartz inequality in the second inequality and \eqref{f1k1,1} in the last inequality.

We now handle $VI$ of \eqref{hess6}. Denote by $\lambda=(0,\lambda_2,\dots,\lambda_n)$ Due to Lemma \ref{con}, we have
\begin{equation}\label{7}
    \begin{aligned}
-\sum_{\gamma,\g\geq2}\lambda_{\gamma}\ddot{F}^{\gamma\gamma\g\g}\lambda_{\g}= &  -\s_k\partial^2_{\lambda}(\log \s_k) \\
=&-\s_k\partial^2_{\lambda}(\log q_k)-\s_k\partial^2_{\lambda}(\log q_{k-1})-\cdots-\s_k\partial^2_{\lambda}(\log \s_1)\\
\geq&-\s_{k-1}\partial^2_{\lambda}q_k,
    \end{aligned}
\end{equation}
which is positive by concavity. We shall give a positive lower bound for the last term of \eqref{7} depending on $|\lambda|$. By Lemma \ref{lem,HG}
\begin{equation}\label{clambdasquare}
    \begin{aligned}
      -\sum_{\gamma,\g\geq2}\lambda_{\gamma}\ddot{F}^{\gamma\gamma\g\g}\lambda_{\g}\geq -\s_{k-1}\partial^2_{\lambda}q_k\geq C\sum_{i=2}^{k-1}\k_1\cdots\hat{\k_i}\cdots\k_{k-1}\left|[\lambda]_{1,\dots\hat{i},\dots,k-1}^{\perp}\right|^2\geq \frac{C_6}{N_2}\k_1|\lambda|^2
    \end{aligned}
\end{equation}
since  $(\lambda_2,\lambda_3,\dots,\lambda_n)\perp(\k_2,\k_3,\dots,\k_n)$ by the definition of \eqref{perp} and the assumption in this case. This positive term is useful.

By using \eqref{lambda}, we have 
\begin{equation}\label{eq11pp}
 \sum_{p=2}^n\ddot{F}^{11pp}\lambda_{p}=\sum_{p=2}^n\frac{\dot{F}^{pp}-\s_{k-1}(\k|1,p)}{\k_1}\lambda_{p}=-\sum_{p=2}^n\frac{\s_{k-1}(\k|1,p)}{\k_1}\lambda_p.
\end{equation}
Together with \eqref{upperboundforsigmak-1} and \eqref{sk-1,1}, $III$ of \eqref{hess6} becomes
\begin{equation}\label{b1}
\begin{aligned}
2ah_{111}\sum_{q=2}^n\ddot{F}^{11qq}\lambda_{q}\geq& -2|ah_{111}|\frac{CN_2^{k-2}}{\k_1}|\lambda|\\\geq& -a^2\frac{C^2N_2^{2(k-2)}h_{111}^2}{\epsilon\k_1^3}-\epsilon    \k_1|\lambda|^2\\
\geq& -Ca^2\frac{N_2^{2(k-2)}\dot{F}^{11}h_{111}^2}{\epsilon\k_1^2}-\epsilon    \k_1|\lambda|^2.
\end{aligned}
\end{equation}
where we use Lemma \ref{fk1} in the last inequality.
In all, plugging \eqref{a1}, \eqref{d}, \eqref{a2},
\eqref{b2}, \eqref{clambdasquare} and \eqref{b1} into \eqref{hess6} yields 
\begin{equation}\label{ppqq3}
    \begin{aligned}
      -\ddot{F}^{pp,qq}\xi_p\xi_q&+\f 2-\frac{C(A)}{\k_1}\r\sum_{p\neq 1}\frac{F^{pp}\xi_p^2}{\k_1}\\
      \geq& (1+a)^2\frac{k(k-1)Fh_{111}^2}{\k_1^2}- C(a^2+1)\f 1+\frac{CN_2^{2(k-2)}}{\epsilon}\r\frac{\dot{F}^{11}h_{111}^2}{\k_1^2}+\frac{C_6}{N_2}\k_1|\lambda|^2\\
&+2a\f a-(1+a)\frac{kF}{\k_1\dot{F}^{11}}\r\frac{\dot{F}^{11}h_{111}^2}{\k_1}-2\epsilon\k_1|\lambda|^2\\ &+2\f1-\frac{C(A)}{\k_1}\r\sum_{p\neq 1}\frac{F^{pp}\xi_p^2}{\k_1}\\
      \geq &\f2a^2+\f(1+a)^2k(k-1)-2a(1+a)k\r\frac{F}{\k_1\dot{F}^{11}}\r\frac{\dot{F}^{11}h_{111}^2}{\k_1}\\
      &+\frac{C_6}{2N_2}\k_1|\lambda|^2-C_4(a^2+1)N_2^{2(k-1)}\frac{\dot{F}^{11}h_{111}^2}{\k_1^2}+2\f1-\frac{C(A)}{\k_1}\r\sum_{p\neq 1}\frac{F^{pp}\xi_p^2}{\k_1}.
    \end{aligned}
\end{equation}
Here we choose $\epsilon=\frac{C_6}{4N_2}$ in the last inequality of \eqref{ppqq3} above such that 
\[2\epsilon\k_1|\lambda|^2=\frac{C_6}{2N_2}\k_1|\lambda|^2. \]

The following proposition improves Proposition \ref{k1f11,f}.
\begin{prop}\label{sigma_kkappa1negative}
   Under \textbf{Assumption 1} and Case 2(b), if $\kappa_1$  is very large (depending on $N_2$), then $\sigma_k(\kappa|1)$ is negative and uniformly away from zero, i.e., there exists $\gamma=\min\lbrace1,\frac{\d_0^k}{2k}\rbrace>0$ such that
    \[\s_k(\kappa|1)\leq -\gamma \s_k.\]
    Moreover,
    \[\left|\sum_{p\geq 2}F^{pp}\kappa_p^2+k\kappa_1\sigma_k(\kappa|1)\right|\leq  C(N_2).\]
    
\end{prop}
\begin{proof}

From Lemma \ref{k1}, we obtain  
\begin{align}\label{fnn}
\d_0^{k-2}\k_1\leq \s_{k-1}=\k_n\s_{k-2}(\k|n)+\s_{k-1}(\k|n)\leq  \s_{k-1}(\k|n).
\end{align}
Thus
\begin{equation}\label{sigmakkappapsqure}
 \begin{aligned}
\d_0^k\k_1\leq\dot{F}^{nn}\k_n^2 \leq \sum_{p=2}^n \dot{F}^{pp}\k_p^2 =& \s_1 \s_k -(k+1) \s_{k+1}-\dot{F}^{11} \k_1^2\\
    =& \s_1(\k|1) \s_k-k \k_1\s_k(\k|1)-(k+1)\s_{k+1}(\k|1)\\
    \leq &CN_2^{k-1}-k \k_1\s_k(\k|1)
\end{aligned}   
\end{equation}
where we use Lemma \ref{f11k12} and \eqref{sk}.

Thus, due to \eqref{sigmakkappapsqure}, if
\begin{align}
    \k_1> C_5N_2^{k-1},
\end{align}
then there exists a $\gamma=\min\lbrace1,\frac{\d_0^k}{2k}\rbrace>0$ such that
\begin{align*}
 \s_k(\k|1)\leq -\gamma\sigma_k<0.
\end{align*}
Note that $\sigma_k$ is constant. The second statement follows from \eqref{sigmakkappapsqure}. In fact we have from \eqref{sigmakkappapsqure}
\begin{equation}\label{sigmakkappapsqure,l}
    \sum_{p=2}^n \dot{F}^{pp}\k_p^2\geq -CN_2^{k-1}-k \k_1\s_k(\k|1).
\end{equation}
\end{proof}
Hence,
\begin{align}\label{F11lowerbound}
    \dot{F}^{11}=\frac{\s_k-\s_k(\k|1)}{\k_1}\geq \frac{\s_k}{\k_1}+\frac{\gamma \sigma_k}{\k_1}.
\end{align}
Combining \eqref{f1k1,1} and \eqref{F11lowerbound}, we obtain
\begin{align}\label{boundsforF11kappa1}
    0<\frac{C}{N_2^{k-2}}\leq\frac{F}{\dot{F}^{11}\k_1}\leq 1-\frac{\gamma}{2}.
\end{align}
By calculation, we conclude that for $3\leq k\leq n-2$, we either have 
\begin{align*}
     \textbf{C}:=2a^2+\f(1+a)^2k(k-1)-2a(1+a)k\r\frac{F}{\k_1\dot{F}^{11}}\geq 2a^2
\end{align*}
if the coefficient of $F/\k_1\dot{F}_{11}$ is non-negative, or by proposition \ref{k1f11,f}, \textbf{C} is greater than 
\begin{equation}\label{coe,3}
  (k^2-3k+2)a^2+(2k^2-4k)a+k^2-k\geq \frac{k+1}{k}+\frac{k-2}{k^3-k^2-k-1}a^2\geq \frac{k+1}{k}+c(k)a^2.
\end{equation}

Suppose that the coefficient of $F/\k_1\dot{F}_{11}$ is negative. It follows from \eqref{ppqq3} that
      \begin{equation}\label{ppqq2}
          \begin{aligned}
        -\ddot{F}^{pp,qq}\xi_p\xi_q+\f 2-\frac{C(A)}{\k_1}\r\sum_{p\neq 1}\frac{F^{pp}\xi_p^2}{\k_1}
        \geq &\textbf{C}\cdot \frac{\dot{F}^{11}h_{111}^2}{\k_1}
      -C_4(a^2+1)N_2^{2(k-1)}\frac{\dot{F}^{11}h_{111}^2}{\k_1^2}\\
        \geq &\f\frac{k+1}{k}+c(k)a^2-C_4\k_1^{-\frac{1}{k}}-C_4 a^2\k_1^{-\frac{1}{k}}\r\frac{\dot{F}^{11}h_{111}^2}{\k_1}\\
 \geq &\frac{\dot{F}^{11}h_{111}^2}{\k_1}.
      \end{aligned}   
      \end{equation}
      by assuming $\k_1\geq N_2^{2k}$ in the first inequality and $\k_1\geq \max\lbrace \f\frac{C_4}{c(k)}\r^k, (kC_4)^k\rbrace $.
      This proves the \textbf{main claim}.

Suppose that the coefficient of $F/\k_1\dot{F}_{11}$ is non-negative and $\textbf{C}\geq 2a^2$.
 In fact, we can assume that there exists some constant $C_1>0$ such that
\begin{align}\label{lambda,u}
    |\lambda_p|\leq \sqrt{\frac{2C_1}{C_6}}\frac{N_2^{\frac{k-1}{2}}(|a|+1)|h_{111}|}{\k_1^{\frac{3}{2}}},\quad \forall\ p=2,\dots,n.
\end{align}
Otherwise by \eqref{clambdasquare} and \eqref{boundsforF11kappa1},
\begin{align}\label{otherwisecase}
     -\sum_{p,q\geq2}\lambda_{p}\ddot{F}^{pp,qq}\lambda_{q}\geq \frac{C_6(\d_0)}{N_2}\k_1|\lambda|^2\geq 4C_1N_2^{k-2}(a^2+1)\frac{h_{111}^2}{\k_1^2}> 4(a^2+1)\frac{\dot{F}^{11}h_{111}^2}{\k_1}
\end{align}
which contributes a desired positive term. 
Then \eqref{ppqq3} yields that
\begin{equation}
\begin{aligned}
    -\ddot{F}^{pp,qq}\xi_p\xi_q+2\sum_{p\neq 1}\frac{F^{pp}\xi_p^2}{(\k_1+1-\k_p)}
     \geq& \frac{C_6}{2N_2}\k_1|\lambda|^2-C_4(a^2+1)N_2^{2(k-1)}\frac{\dot{F}^{11}h_{111}^2}{\k_1^2}\\
     \geq& \f2(a^2+1)-C_4(a^2+1)\k_1^{-\frac{1}{k}}\r\frac{\dot{F}^{11}h_{111}^2}{\k_1}\\
>&\frac{\dot{F}^{11}h_{111}^2}{\k_1}
\end{aligned}
\end{equation}
where we assume $\k_1\geq N_2^{2k}$ in the second inequality and $\k_1\geq C_4^k$ in the last inequality. This implies the \textbf{main claim}.

Hence we assume that \eqref{lambda,u} holds. By definition \eqref{xi} for any $2\leq p\leq n$,
\begin{align}\label{plugxiinlastterm}
 \xi_p^2=&\f(1+a)\frac{\k_ph_{111}}{\k_1}+\frac{b\dot{F}^{pp}}{|\b|^2}+\lambda_p\r^2\nonumber\\
 \geq &(1+a)^2\frac{\k_p^2h_{111}^2}{\k_1^2}-2\left|(1+a)\f\frac{b\dot{F}^{pp}}{|\b|^2}+\lambda_p\r\frac{\k_ph_{111}}{\k_1}\right|.
\end{align}
Recall that $\k_2\leq N_2$ in the assumption of Case 2(b).
Combining Proposition \ref{f11est} and \eqref{lambda,u}, we have 
\begin{equation}\label{xi,1}
    \begin{aligned}
        \xi_p^2\geq &(1+a)^2\frac{\k_p^2h_{111}^2}{\k_1^2}-C(a^2+1)\frac{N_2h_{111}^2}{\k_1^3}-C(a^2+1)\frac{N_2^{\frac{k+1}{2}}h_{111}^2}{\k_1^{\frac{5}{2}}}\\
        \geq&(1+a)^2\frac{\k_p^2h_{111}^2}{\k_1^2}-2C(a^2+1)\frac{N_2^{\frac{k+1}{2}}h_{111}^2}{\k_1^{\frac{5}{2}}}
    \end{aligned}
\end{equation}
by assuming $\k_1>1$.

As for any $2\leq p\leq n$, 
\begin{equation}\label{fpp,est}
  \dot{F}^{pp}\leq C(n,k)\k_1\k_2\cdots\k_{k-1}\leq C(n,k)\s_{k-1}\leq C(n,k)N_2^{k-2}\k_1.  
\end{equation}

Then combining Proposition \ref{f11est}, \eqref{sk-1,1} and \eqref{sigmakkappapsqure,l}, we have
\begin{equation}\label{7,1}
\begin{aligned}
    2\f1-\frac{c}{\k_1}\r\sum_{p\neq 1}\frac{\dot{F}^{pp}\xi_p^2}{\k_1}\geq& 2\f1-\frac{c}{\k_1}\r\frac{(1+a)^2h_{111}^2}{\k_1^3}\sum_{p\geq 2}\dot{F}^{pp}\k_p^2-C(a^2+1)\frac{N_2^{\frac{k+1}{2}}\sum_{p\geq 2}\dot{F}^{pp}h_{111}^2}{\k_1^{\frac{7}{2}}}\\
    \geq& 2\f1-\frac{c}{\k_1}\r k\frac{(1+a)^2h_{111}^2}{\k_1^2}(-C\frac{N_2^{k-1}}{\k_1}+\k_1\dot{F}^{11}-\s_k)-\frac{C(a^2+1)N_2^{\frac{3k-3}{2}}h_{111}^2}{\k_1^{\frac{5}{2}}}\\
    \geq& 2k\frac{(1+a)^2h_{111}^2}{\k_1^2}(\k_1\dot{F}^{11}-\s_k)- C(a^2+1)N_2^{\frac{3k-3}{2}}\frac{\dot{F}^{11}h_{111}^2}{\k_1^\frac{3}{2}}.
    \end{aligned}    
\end{equation}

Plugging \eqref{7,1} into \eqref{ppqq3}, we obtain the conclusion in this case,
    \begin{align}
      -\ddot{F}^{pp,qq}\xi_p\xi_q&+2\f1-\frac{c}{\k_1}\r\sum_{p\neq 1}\frac{F^{pp}\xi_p^2}{\k_1}\nonumber\\
\geq&\f2a^2+2k(1+a)^2+\f(1+a)^2k(k-1)-2a(1+a)k-2k(1+a)^2\r t\r\frac{\dot{F}^{11}h_{111}^2}{\k_1}\nonumber\\
    &-C_7(a^2+1)N_2^{2(k-1)}\frac{\dot{F}^{11}h_{111}^2}{\k_1^\frac{3}{2}}\label{ppqq5}
    \end{align}
where we assume $F=t\k_1\dot{F}^{11}$ in \eqref{ppqq5}. 

If the coefficient of $t$ in \eqref{ppqq5} is negative, by calculation in \eqref{ppqq2}, \eqref{ppqq5} implies 
 \begin{align}
      -\ddot{F}^{pp,qq}\xi_p\xi_q&+2\f1-\frac{c}{\k_1}\r\sum_{p\neq 1}\frac{F^{pp}\xi_p^2}{\k_1}\nonumber\\
\geq&\f2a^2+\f(1+a)^2k(k-1)-2a(1+a)k\r\r\frac{\dot{F}^{11}h_{111}^2}{\k_1}\nonumber\\
      &-C_7(a^2+1)N_2^{2(k-1)}\frac{\dot{F}^{11}h_{111}^2}{\k_1^\frac{3}{2}}\nonumber\\
      \geq&\f \frac{k+1}{k}+c(k)a^2\r\frac{\dot{F}^{11}h_{111}^2}{\k_1}\nonumber\\
      &-C_7(a^2+1)N_2^{2(k-1)}\frac{\dot{F}^{11}h_{111}^2}{\k_1^\frac{3}{2}}\nonumber\\
      > &\frac{\dot{F}^{11}h_{111}^2}{\k_1}
      \end{align}
      by choosing $\k_1\geq \max\left\lbrace N_2^{4k}, (kC_7)^{2k},\f\frac{C_7}{c(k)}\r^{2k}\right\rbrace$.

If the coefficient of $t$ in \eqref{ppqq5} is positive, \eqref{ppqq5} implies   
 \begin{align}
      -\ddot{F}^{pp,qq}\xi_p\xi_q&+2\f1-\frac{c}{\k_1}\r\sum_{p\neq 1}\frac{F^{pp}\xi_p^2}{\k_1}\nonumber\\
    \geq &\f2a^2+2k(1+a)^2\r\frac{\dot{F}^{11}h_{111}^2}{\k_1}-C_7(a^2+1)N_2^{2(k-1)}\frac{\dot{F}^{11}h_{111}^2}{\k_1^\frac{3}{2}}\nonumber\\
    \geq &\f \frac{3}{2}a^2+4(1+a)^2-\frac{C_7}{\k_1^{\frac{1}{2k}}}+\f\frac{1}{2}-\frac{C_7}{\k_1^{\frac{1}{2k}}}\r a^2\r\frac{\dot{F}^{11}h_{111}^2}{\k_1}\\
    > &\frac{\dot{F}^{11}h_{111}^2}{\k_1}
    \end{align}
     by choosing $\k_1\geq \max\left\lbrace N_2^{4k}, (2C_7)^{2k}\right\rbrace$. 
This proves the \textbf{Main claim}.

\subsection{Assumption 2:}
We suppose that $2\leq l\leq k-2$ and $4\leq k\leq n-2$.\\

\begin{prop}\label{relationbetweens_k+1ands_k-1} Under assumptions of Case $2$ and Case $2(b)$, there exists large $N_2$ in \textbf{Assumption 2} such that
    \[-\s_{k+1}\geq C(n,k,a_1)\s_{k-1}\]
    where $a_1$ is the lower bound appearing in Lemma \ref{bnu}. In particular, $\s_{k+1}$ is negative. 
\end{prop}
\begin{proof}
    Since $\sigma_1\leq n\k_1$ and $\k_k\geq C|\k_n|$, using Lemma \ref{k1} we have
    \begin{equation}\label{sk-1,n2}
        \s_{k-1}> \k_1\cdots\k_{k-1}\geq C(\d_0)\k_1\cdots\k_l \geq C\k_1 N_2^{2\ell-2}.
    \end{equation}
    Then $\sigma_{k-1}\geq C_2N_2^{2\ell-2}\sigma_1.$ 
    Recall \eqref{sk-1} 
    \begin{align}
 (n-k+1)\s_{k-1}<\frac{16}{a_1^2}\f\s_1\s_k-(k+1)\s_{k+1}\r.   
\end{align}
We choose $N_2\geq \f\frac{16\s_k}{a_1^2}\r^{\frac{1}{2l-2}}$, such that \eqref{sk-1,n2} yields that
\begin{equation}
 \begin{aligned}
  -\s_{k+1}>& \frac{a_1^2}{16(k+1)}(n-k+1)\s_{k-1}-\frac{\s_1\s_k}{k+1}\\
  \geq &\frac{(n-k)a_1^2}{16(k+1)}\s_{k-1}+ \frac{C_2a_1^2N_2^{2\ell-2}}{16(k+1)}\sigma_1-\frac{\s_1\s_k}{k+1}\\
  \geq& \frac{(n-k)a_1^2}{16(k+1)}\s_{k-1},
\end{aligned}   
\end{equation}
which prove the proposition.
\end{proof}
Recall \eqref{sk}.
For any $i\leq l$,
\begin{equation}\label{kifi}
\begin{aligned}
    \k_i\dot{F}^{ii}&=\s_k-\s_k(\k|i)=\s_k-\frac{\s_{k+1}}{\k_i}+\frac{\s_{k+1}(\k|i)}{\k_i}.
\end{aligned}    
\end{equation}
By Lemma \ref{k1} and Lemma \ref{lemn}, Proposition \ref{relationbetweens_k+1ands_k-1} implies that
\begin{align}\label{sk+1,i}
   |\s_{k+1}(\k|i)|\leq C\frac{\k_1\cdots\k_k^3}{\k_i}\leq C(A)\frac{\k_1\cdots\k_{k-1}}{\k_i}\leq C(A,a_1)\frac{-\s_{k+1}}{\k_i}.
\end{align}
Suppose $i>1$. 
\begin{align}
    \s_k\leq C(\d_0)\frac{\k_1\cdots\k_{k-1}}{\k_i^2}\leq C(\d_0)\frac{\s_{k-1}}{\k_i^2}\leq-C(\d_0,a_1)\frac{\s_{k+1}}{\k_i^2}.
\end{align}
Suppose $i=1$. Then for $3\leq k\leq n-2$ by Proposition \ref{relationbetweens_k+1ands_k-1} and \eqref{sk-1,n2},
\begin{align}\label{sk,n2}
    \s_k\leq C(\d_0)\frac{\s_{k-1}}{\k_1N_2^{2l-2}}\leq C(\d_0,a_1)\frac{-\s_{k+1}}{\k_1N_2^2}.
\end{align}
Thus we have
\begin{align}\label{k1f1}
   \f1-\frac{c_2}{N_2^2}\r\frac{-\s_{k+1}}{\k_i}\leq \k_i\dot{F}^{ii}\leq \f1+\frac{c_2}{N_2^2}\r\frac{-\s_{k+1}}{\k_i}.
\end{align}
Or equivalently, for any $1\leq i\leq l$ there exists small $\epsilon(N_2)>0$ such that 
\begin{align}\label{ki2fii}
    -\f1-\epsilon(N_2)\r\s_{k+1}\leq \dot{F}^{ii}\k_i^2\leq -\f1+\epsilon(N_2)\r\s_{k+1}.   
\end{align}
where $\epsilon(N_2)=\frac{c_2}{N_2^2}. $
Using \eqref{sk-1,n2}, \eqref{ki2fii} and Proposition \ref{relationbetweens_k+1ands_k-1}, we have 
\begin{align}\label{kifii}
    \k_i\dot{F}^{ii}\geq\frac{-\s_{k+1}}{2\k_i}\geq \frac{C}{2}\frac{\s_{k-1}}{\k_i}\geq c(\d_0)N_2^2.
\end{align}
Using Proposition \ref{relationbetweens_k+1ands_k-1} and \eqref{ki2fii}, we have  
\begin{align}\label{kifii,u}
    \k_i\dot{F}^{ii}\leq \frac{-3\s_{k+1}}{2\k_i}\leq \frac{3}{2}C(A)\frac{\s_{k-1}}{\k_i}.
\end{align}
Here we choose $N_2>\sqrt{2c_2}$ such that $\epsilon({N_2})<\frac{1}{2}$ in \eqref{kifii} and \eqref{kifii,u}.
Recall \eqref{b}.
\begin{align}
    b=\f a-(1+a)\frac{kF}{\k_1\dot{F}^{11}}\r\dot{F}^{11}h_{111}.
\end{align}
Then by \eqref{kifii}, we have
\begin{align}\label{b,est}
   a\dot{F}^{11}h_{111}-\frac{C(|a|+1)\dot{F}^{11}|h_{111}|}{N_2^2} \leq b\leq  a\dot{F}^{11}h_{111}+\frac{C(|a|+1)\dot{F}^{11}|h_{111}|}{N_2^2}. 
\end{align}
$II$ of \eqref{hess6} becomes
\begin{equation}\label{f1pfp}
   \begin{aligned}
   2ah_{111}\frac{b}{|\b|^2}\sum_{\a=2}^n\ddot{F}^{11\a\a}\dot{F}^{\a\a}=& 2ah_{111}\frac{b}{|\b|^2}\sum_{\a=2}^n\f\frac{\dot{F}^{\a\a}}{\k_1}-\frac{\s_{k-1}(\k|1,\a)}{\k_1}\r\dot{F}^{\a\a} \\
   \geq & 2ab\frac{h_{111}}{\k_1}-2 \frac{|abh_{111}|}{\k_1} \sum_{\a=2}^n\frac{|\s_{k-1}(\k|1,\a)|\dot{F}^{\a\a}}{|\b|^2}\\
   \geq&\f2a^2-\frac{C(a^2+1)}{N_2^2}\r\frac{\dot{F}^{11}h_{111}^2}{\k_1}-C(a^2+1)\frac{\dot{F}^{11}h_{111}^2}{\k_1^2} \\
   \geq&2a^2\frac{\dot{F}^{11}h_{111}^2}{\k_1}-\frac{C(a^2+1)}{N_2^2}\frac{\dot{F}^{11}h_{111}^2}{\k_1}.
\end{aligned} 
\end{equation}
Here we use Proposition \ref{betasigmak-1kappa1} and \eqref{b,est} in the last inequality of \eqref{f1pfp}. 

Using Proposition \ref{betasigmak-1kappa1} and \eqref{kifii,u} we get for any $2\leq i\leq \ell$
\begin{align}
    \left|\frac{b\dot{F}^{ii}}{|\b|^2}\right|=&\frac{\f a-(1+a)\frac{kF}{\k_1\dot{F}^{11}}\r\dot{F}^{11}\dot{F}^{ii}h_{111}}{|\b|^2}\leq\frac{C(|a|+1)\dot{F}^{11}\dot{F}^{ii}|h_{111}|}{\s_{k-1}^2}\nonumber\\
    \leq & \frac{C(A)(|a|+1)|h_{111}|}{\k_1^2\k_i^2}\leq \frac{C(A)(|a|+1)\k_i|h_{111}|}{\k_1N_2^6}\leq \frac{(|a|+1)\k_i|h_{111}|}{\k_1}
\end{align}
by the assumption in this case that $\k_i>N_2^2>1$.
From $\xi_i=(1+a)\frac{\k_ih_{111}}{\k_1}+\frac{b\dot{F}^{ii}}{|\b|^2}+\lambda_i$ and \eqref{k1f1}, we assume there exists some constant $C>8$, such that for $i=2,\cdots, l$
\begin{align}\label{lambdai}
    |\lambda_i|\leq C(|a|+1) \frac{\k_i|h_{111}|}{\k_1}.
\end{align}
Otherwise, we have
\begin{equation}
    |\xi_i|\geq \frac{|\lambda_i|}{2}\geq \frac{C}{2}(|a|+1) \frac{\k_i|h_{111}|}{\k_1}.
\end{equation}
Hence \eqref{hess6} becomes
 \begin{align}
    -\ddot{F}^{pp,qq}\xi_p\xi_q+\f 2-\frac{C(A)}{\k_1}\r\sum_{p\neq 1}\frac{F^{pp}\xi_p^2}{\k_1}\geq& \frac{3}{2}\sum_{2\leq i\leq l}\frac{F^{ii}\xi_i^2}{\k_1}\geq \frac{3}{8}\sum_{2\leq i\leq n}\frac{F^{ii}\lambda_{i}^2}{\k_1}\geq \frac{\dot{F}^{11}\k_1\lambda_i^2}{8\k_i^2}\label{711}\nonumber\\
    \geq& \frac{C^2(a^2+1)\dot{F}^{11}h_{111}^2}{32\k_1}
    \geq 2(a^2+1)\frac{\dot{F}^{11}h_{111}^2}{\k_1}
\end{align}  
where we use \eqref{kifii} and \eqref{kifii,u} in the  inequality in the third inequality of \eqref{ki2fii}. This implies the \textbf{main claim}.
Hence we assume \eqref{lambdai} holds below. 
By \eqref{eq11pp}, $III$ of \eqref{hess6} becomes

\begin{equation}
 \begin{aligned}
\sum_{\gamma=2}^n\ddot{F}^{11\gamma\gamma}\lambda_{\gamma}=&\sum_{\gamma=2}^n\frac{-\s_{k}(\k|\gamma)+\s_{k}(\k|1 \gamma)}{\k_1^2}\lambda_{\gamma}
=\sum_{\gamma=2}^n\frac{\dot{F}^{\gamma\gamma}\k_\gamma-\s_k+\s_{k}(\k|1 \gamma)}{\k_1^2}\lambda_{\gamma}\\
\geq& \sum_{i=2}^l\f\frac{\dot{F}^{ii}\k_i}{\k_1^2}\lambda_{i}-\frac{\s_k+|\s_{k}(\k|1 i)|}{\k_1^2}|\lambda_i|\r+\sum_{p=l+1}^n\f\frac{\dot{F}^{pp}\k_p}{\k_1^2}\lambda_{p}-\frac{\s_k+|\s_{k}(\k|1 p)|}{\k_1^2}|\lambda_p|\r.
\end{aligned}   
\end{equation}
For $2\leq i\leq l$ we have by Proposition \ref{relationbetweens_k+1ands_k-1} and \eqref{kifii} 
\begin{equation}\label{sk,1i}
   |\s_k(\k|1i)|\leq C(n,k)\frac{\k_1\cdots\k_k^3}{\k_1\k_i}\leq -C(n,k,A,a_1)\frac{\s_{k+1}}{\k_1\k_i}\leq C\frac{\dot{F}^{11}\k_1}{\k_i}.
\end{equation}
For $l+1\leq p\leq n$, similar to the proof of \eqref{sk,1i} we have 
\begin{equation}\label{sk,1p}
   |\s_k(\k|1p)|\leq C(n,k)\frac{\k_1\cdots\k_k^2}{\k_1}\leq -C(n,k,A,a_1)\frac{\s_{k+1}}{\k_1}\leq C\dot{F}^{11}\k_1.
\end{equation}
Besides, from Proposition \ref{relationbetweens_k+1ands_k-1}, \eqref{sk-1,n2} and \eqref{kifii} by choosing $N_2$ large depending on $\epsilon(N_2)$ there exists $c(N,a_1)>0$ such that  
\begin{equation}\label{sk,2}
    \begin{aligned}
        \frac{\dot{F}^{11}}{\k_1\k_i}\geq -\frac{1}{2}\frac{\s_{k+1}}{\k_1^3\k_i}\geq \frac{C\s_{k-1}}{\k_1^3\k_i}>\frac{C\k_1\k_2\cdots\k_{k-1}}{\k_1^3\k_i}\geq \frac{C(\d_0)}{\k_1^2}\geq c\frac{\s_k}{\k_1^2}
    \end{aligned}
\end{equation}
for any $2\leq i\leq l$.
Furthermore, by \eqref{ki2fii} we have 
\begin{align}
    \dot{F}^{ii}\k_i^2\leq -\f1+\epsilon(N_2)\r\s_{k+1}
\end{align}
and 
\begin{equation}
    |\dot{F}^{ii}\k_i^2-\dot{F}^{11}\k_1^2|\leq 2\epsilon(N_2)\s_{k+1}.
\end{equation}
Besides, by \eqref{sk} and \eqref{kifi} 
\begin{equation}\label{fpkp}
    |\dot{F}^{pp}\k_p|\leq |\s_{k+1}|+|\s_{k+1}(\k|p)|\leq C\k_1\cdots\k_{k-1}\k_k^2\leq C(A)\s_{k-1}\leq -C(A,a_1)\s_{k+1}\leq -2C\dot{F}^{11}\k_1^2.
\end{equation}
By \eqref{kifi}, \eqref{lambdai} and \eqref{fpkp}, we have 
\begin{equation}\label{b4}
 \begin{aligned}
2ah_{111}\sum_{\gamma=2}^n\ddot{F}^{11\gamma\gamma}\lambda_{\gamma}\geq&2ah_{111}\sum_{i=2}^l\f\frac{\dot{F}^{ii}\k_i}{\k_1^2}\lambda_{i}-c\frac{\dot{F}^{11}}{\k_1\k_i}|\lambda_i|\r+2ah_{111}\sum_{p=l+1}^n\f\frac{\dot{F}^{pp}\k_p}{\k_1^2}\lambda_{p}-c\frac{\dot{F}^{11}}{\k_1}|\lambda_p|\r\\
\geq &2ah_{111}\sum_{i=2}^l\frac{\dot{F}^{11}}{\k_i}\lambda_{i}-8\epsilon(N_2)\frac{\dot{F}^{11}|ah_{111}\lambda_{i}|}{\k_i}\\
&-2c|ah_{111}|\sum_{i=2}^l \frac{\dot{F}^{11}}{\k_1\k_i}|\lambda_i|-2c'|ah_{111}|\sum_{p=l+1}^n\dot{F}^{11}|\lambda_p|\\
\geq&2a\sum_{i=2}^l\frac{\dot{F}^{11}h_{111}}{\k_i}\lambda_{i}-\frac{C}{{N_2}^2}(a^2+1)\frac{\dot{F}^{11}h_{111}^2}{\k_1}-C(a^2+1)\frac{\dot{F}^{11}h_{111}^2}{\k_1^2}\\
&-\epsilon\sum_{p=l+1}^n\dot{F}^{11}\k_1^2\lambda_p^2-\frac{Ca^2}{\epsilon}\frac{\dot{F}^{11}h_{111}^2}{\k_1^2}
\end{aligned}   
\end{equation}
where we use Cauchy-Schwartz inequality in the last inequality.

By Lemma \ref{lem,HG} and \eqref{kifii,u}, $VI$ of \eqref{hess6} becomes
\begin{align}
     -\sum_{p,q\geq2}\lambda_{p}\ddot{F}^{pp,qq}\lambda_{q}\geq C\frac{\s_{k-1}}{N_2^2}|\lambda^l|^2\geq \frac{C}{N_2^2}\dot{F}^{11}\k_1^2|\lambda^l|^2\label{71}
\end{align}

For $V$ of \eqref{hess6}, we have
\begin{equation}\label{b22}
\begin{aligned}
    -2\frac{b}{|\b|^2}\sum_{p,\gamma\geq2}\dot{F}^{pp}\ddot{F}^{pp\gamma\gamma}\lambda_{\gamma}= 
    &-2\frac{b}{|\b|^2}\sum_{l+1\leq p\leq n,2\leq i\leq l}\dot{F}^{pp}\f\frac{\dot{F}^{pp}}{\k_i}-\frac{\s_{k-1}(\k|ip)}{\k_i}\r\lambda_i\\
    &-2\frac{b}{|\b|^2}\sum_{2\leq\gamma\leq n,l+1\leq q\leq n}\dot{F}^{\gamma\gamma}\ddot{F}^{\gamma\gamma,qq}\lambda_q
   -2\frac{b}{|\b|^2}\sum_{2\leq j\neq i\leq l}\dot{F}^{jj}\ddot{F}^{jj,ii}\lambda_i.
\end{aligned}    
\end{equation}
Since by Proposition \ref{betasigmak-1kappa1} for any $2\leq\gamma\leq n$, $l+1\leq q\leq n$
\begin{equation}
    \dot{F}^{\gamma\gamma}\leq c\s_{k-1}\leq c'|\b|
\end{equation}
and 
\begin{equation}\label{6,223}
    \dot{F}^{\gamma\gamma,pp}\leq c\s_{k-1}\leq c'|\b|,
\end{equation}
we have by \eqref{b,est}
\begin{equation}\label{6,22}
\begin{aligned}
        -2\frac{b}{|\b|^2}\sum_{2\leq\gamma\leq n,l+1\leq q\leq n}\dot{F}^{\gamma\gamma}\ddot{F}^{\gamma\gamma,qq}\lambda_q\geq&-2\sum_{l+1\leq q\leq n}C(|a|+1) \dot{F}^{11}|h_{111}\lambda_p|\\
    \geq&-\epsilon\sum_{p=l+1}^n\dot{F}^{11}\k_1^2\lambda_p^2-\frac{C^2(a^2+1)}{\epsilon}\frac{\dot{F}^{11}h_{111}^2}{\k_1^2}.
\end{aligned}
\end{equation}
Since by Proposition \ref{betasigmak-1kappa1} and \eqref{kifii,u} for any $2\leq j\neq i\leq l$
\begin{equation}\label{fii,n2}
    \dot{F}^{jj}\leq C\frac{\s_{k-1}}{\k_j^2}\leq C'\frac{|\b|}{\k_j^2}
\end{equation}
and 
\begin{equation}
    \ddot{F}^{jj,ii}\leq C(n,k)\frac{\k_1\cdots\k_k^3}{\k_i\k_j} \leq C'(n,k,A)\frac{\s_{k-1}}{\k_i\k_j}\leq C''(n,k,A) \frac{|\b|}{\k_i\k_j}
\end{equation}
we have by \eqref{b,est} and \eqref{lambdai}
\begin{equation}\label{6,21}
    -2\frac{b}{|\b|^2}\sum_{2\leq j\neq i\leq l}\dot{F}^{jj}\ddot{F}^{jj,ii}\lambda_i\geq -\sum_{2\leq j\neq i\leq l}C(|a|+1)\frac{\dot{F}^{11}|h_{111}\lambda_i|}{\k_i\k_j^3}\geq-C'(a^2+1)\frac{\dot{F}^{11}h_{111}^2}{\k_1N_2^6}. 
\end{equation}
Besides, since by \eqref{fii,n2} for any $l+1\leq p\leq n,2\leq i\leq l$
\begin{align}
  \f1-\frac{C}{N_2^4}\r|\b|^2\leq \sum_{l+1\leq p\leq n}(\dot{F}^{pp})^2=|\b|^2- \sum_{2\leq i\leq l}(\dot{F}^{ii})^2\leq\f1+\frac{C}{N_2^4}\r|\b|^2
\end{align} 
and by Lemma \ref{betasigmak-1kappa1}
\begin{equation}
    |\s_{k-1}(\k|ip)|\leq C(n,k)\frac{\k_1\cdots\k_k}{\k_i}\leq C'(n,k,A)\frac{\s_{k-1}}{\k_i}\leq C''(n,k,A)\frac{|\b|}{\k_i},
\end{equation}
using \eqref{b,est} we have
\begin{equation}\label{6,24}
    \begin{aligned}
    &-2\frac{b}{|\b|^2}\sum_{l+1\leq p\leq n,2\leq i\leq l}\dot{F}^{pp}\f\frac{\dot{F}^{pp}}{\k_i}-\frac{\s_{k-1}(\k|ip)}{\k_i}\r\lambda_i\\
    \geq &-2\frac{b}{|\b|^2}\sum_{l+1\leq p\leq n}(\dot{F}^{pp})^2\sum_{2\leq i\leq l}\frac{\lambda_i}{\k_i}-2C(n,k,A)(|a|+1)\frac{\dot{F}^{11}|h_{111}\lambda_i|}{\k_i^2}\\
      \geq &-2 a\frac{\dot{F}^{11}h_{111}\lambda_i}{\k_i}-\frac{C'(|a|+1)}{N_2^2}\sum_{2\leq i\leq l}\frac{\dot{F}^{11}|\lambda_ih_{111}|}{\k_i}-2C(n,k,A)(|a|+1)\frac{\dot{F}^{11}|h_{111}\lambda_i|}{\k_i^2}\\
      \geq &-2 a\frac{\dot{F}^{11}h_{111}\lambda_i}{\k_i}-\frac{C''(a^2+1)}{N_2^2}\frac{\dot{F}^{11}h_{111}^2}{\k_1}.
\end{aligned}
\end{equation}

Thus $V$ of \eqref{hess6} becomes
\begin{equation}\label{6,2}
    \begin{aligned}
        -2\frac{b}{|\b|^2}\sum_{p,\gamma\geq2}\dot{F}^{pp}\ddot{F}^{pp\gamma\gamma}\lambda_{\gamma}\geq  
    &-2 a\frac{\dot{F}^{11}h_{111}\lambda_i}{\k_i}-\epsilon\sum_{p=l+1}^n\dot{F}^{11}\k_1^2\lambda_p^2\\
    &-\frac{C(a^2+1)}{N_2^2}\frac{\dot{F}^{11}h_{111}^2}{\k_1}-\frac{C^2(a^2+1)}{\epsilon}\frac{\dot{F}^{11}h_{111}^2}{\k_1^2}.
    \end{aligned}
\end{equation}

Then we consider $VII$ of \eqref{hess6} 
\begin{align}\label{7,2}
    2\f1-\frac{c}{\k_1}\r\sum_{\gamma\geq 2}\frac{F^{\gamma\gamma}\xi_{\gamma}^2}{\k_1}=2\f1-\frac{c}{\k_1}\r\sum_{2\leq i\leq l}\frac{F^{ii}\xi_i^2}{\k_1}+2\f1-\frac{c}{\k_1}\r\sum_{l+1\leq p\leq n}\frac{F^{pp}\xi_p^2}{\k_1}
\end{align}

Meanwhile, by \eqref{b4} and \eqref{6,2} we have 
\begin{align}\label{46}
2ah_{111}\sum_{\gamma=2}^n\ddot{F}^{11\gamma\gamma}\lambda_{\gamma}-&2\frac{b}{|\b|^2}\sum_{p,\gamma\geq2}\dot{F}^{pp}\ddot{F}^{pp\gamma\gamma}\lambda_{\gamma}\nonumber\\
\geq&
-2\epsilon\sum_{p=l+1}^n\dot{F}^{11}\k_1^2\lambda_p^2-\frac{C(a^2+1)}{N_2^2}\frac{\dot{F}^{11}h_{111}^2}{\k_1}-\frac{C^2(a^2+1)}{\epsilon}\frac{\dot{F}^{11}h_{111}^2}{\k_1^2}.
\end{align}
Combining \eqref{a1}, \eqref{d}, \eqref{f1pfp}, \eqref{71}, \eqref{7,2} and \eqref{46}, \eqref{hess6} becomes
\begin{equation}\label{hess7}
  \begin{aligned}
      -\ddot{F}^{pp,qq}\xi_p\xi_q+&\f 2-\frac{C(A)}{\k_1}\r\sum_{p\neq 1}\frac{F^{pp}\xi_p^2}{\k_1}\\
      \geq&
(1+a)^2\frac{k(k-1)Fh_{111}^2}{\k_1^2}+2a^2\frac{\dot{F}^{11}h_{111}^2}{\k_1}-\frac{C(a^2+1)}{N_2^2}\frac{\dot{F}^{11}h_{111}^2}{\k_1}- C(a^2+1)\frac{\dot{F}^{11}h_{111}^2}{\k_1^2}\\
&-2\epsilon\sum_{p=l+1}^n\dot{F}^{11}\k_1^2\lambda_p^2+\frac{C_7}{N_2^2}\sum_{p=l+1}^n\dot{F}^{11}\k_1^2\lambda_p^2-\frac{C(a^2+1)}{\epsilon}\frac{\dot{F}^{11}h_{111}^2}{\k_1^2}+2\f1-\frac{c}{\k_1}\r\sum_{p=l+1}^n\frac{F^{pp}\xi_p^2}{\k_1}\\
\geq
&2a^2\frac{\dot{F}^{11}h_{111}^2}{\k_1}-\frac{C(a^2+1)}{N_2^2}\frac{\dot{F}^{11}h_{111}^2}{\k_1}+\frac{C_7}{2N_2^2}\sum_{p=l+1}^n\dot{F}^{11}\k_1^2\lambda_p^2+C'N_2^2(a^2+1)\frac{\dot{F}^{11}h_{111}^2}{\k_1^2}\\
&+2\f1-\frac{c}{\k_1}\r\sum_{l+1\leq p\leq n}\frac{F^{pp}\xi_p^2}{\k_1}
  \end{aligned}  
\end{equation}
by choosing $\epsilon=\frac{C_7}{4N_2^2}$. 
For the same reason in Assumption 1 of case 2(b) above, we claim that there exists some $C>0$ large depending on $n,k,A$ such that
\begin{align}\label{lambdap}
    |(\lambda_{l+1},\dots,\lambda_n)|\leq \frac{CN_2(|a|+1)|h_{111}|}{\k_1^{\frac{3}{2}}}.
\end{align}
Or by \eqref{71}
\begin{align}
     -\sum_{p,q\geq2}\lambda_{p}\ddot{F}^{pp,qq}\lambda_{q}
     \geq& C'C^2(a^2+1)\frac{\dot{F}^{11}h_{111}^2}{\k_1}\geq 2(a^2+1)\frac{\dot{F}^{11}h_{111}^2}{\k_1}.\label{72}
\end{align}
Recall that 
\begin{equation}
    \xi_p^2=\f(1+a)\frac{\k_ph_{111}}{\k_1^2}+\frac{b\dot{F}^{pp}}{|\b|^2}+\lambda_p\r^2.
\end{equation}
Using Proposition \ref{betasigmak-1kappa1}, Proposition \ref{f11est} and \eqref{lambdap}, similar to \eqref{xi,1} we have 
\begin{equation}\label{xi,2}
    \begin{aligned}
        \xi_p^2\geq (1+a)^2\frac{\k_p^2h_{111}^2}{\k_1^2}-C(a^2+1)\frac{N_2^2h_{111}^2}{\k_1^{\frac{5}{2}}}
    \end{aligned}
\end{equation}
and similar to \eqref{7,1} the last term in \eqref{hess7} becomes
\begin{align}\label{73}
    2\f1-\frac{C}{\k_1}\r\sum_{l+1\leq p\leq n}\frac{F^{pp}\xi_p^2}{\k_1}=&2\f1-\frac{C}{\k_1}\r\sum_{l+1\leq p\leq n}\frac{F^{pp}}{\k_1}\f(1+a)\frac{\k_ph_{111}}{\k_1^2}+\frac{b\dot{F}^{pp}}{|\b|^2}+\lambda_p\r^2\nonumber\\
    \geq &2\f1-\frac{C}{\k_1}\r(1+a)^2\sum_{l+1\leq p\leq n}\frac{F^{pp}\k_p^2h_{111}^2}{\k_1^3}-CN_2^2(a^2+1)\frac{\dot{F}^{11}h_{111}^2}{\k_1^{\frac{3}{2}}}.
\end{align}
 By using \eqref{sk}, \eqref{sk+1,i}, \eqref{sk,n2} and \eqref{k1f1}, we calculate the first term \eqref{73} as
 \begin{equation}\label{74}
  \begin{aligned}
  \sum_{l+1\leq p\leq n}F^{pp}\k_p^2=&\s_1\s_k-(k+1)\s_{k+1}-\sum_{1\leq i\leq l}F^{ii}\k_i^2\\
=&\sum_{p=l+1}^n\k_p\s_k+\sum_{i=1}^l\k_i\f\s_k(\k|i)+\k_i\s_{k-1}(\k|i)\r-(k+1-l)\s_{k+1}\\
  &-\sum_{i=1}^l\f\s_{k+1}(\k|i)+\k_i\s_k(\k|i)\r-\sum_{1\leq i\leq l}F^{ii}\k_i^2\\
=&\sum_{p=l+1}^n\k_p\s_k-(k+1-l)\s_{k+1}-\sum_{i=1}^l\s_{k+1}(\k|i)\\
  \geq& \f1-\frac{C}{N_2^2}\r(k+1-l)\dot{F}^{11}\k_1^2.
\end{aligned}   
 \end{equation}

Plugging \eqref{74} into \eqref{hess7}, we have 
\begin{equation}
  \begin{aligned}
   -\ddot{F}^{pp,qq}\xi_p\xi_q+&\f 2-\frac{C(A)}{\k_1}\r\sum_{p\neq 1}\frac{F^{pp}\xi_p^2}{\k_1}\\
   \geq&\f2a^2+2(k+1-l)(1+a)^2-\frac{C(a^2+1)}{N_2^2}\r\frac{\dot{F}^{11}h_{111}^2}{\k_1} -CN_2^2(a^2+1)\frac{\dot{F}^{11}h_{111}^2}{\k_1^{\frac{3}{2}}}\\
   \geq&\f2a^2+2(k+1-l)(1+a)^2-\frac{C_6(a^2+1)}{N_2^2}\r\frac{\dot{F}^{11}h_{111}^2}{\k_1}\\
   \geq&\f\frac{1}{2}a^2+\frac{6}{5}-\frac{C_6(a^2+1)}{N_2^2}\r\frac{\dot{F}^{11}h_{111}^2}{\k_1}\\
>&\frac{\dot{F}^{11}h_{111}^2}{\k_1}
\end{aligned}  
\end{equation}
by choosing $\k_1\geq N_2^9$ in the second inequality and $N_2>(5C_6)^{\frac{1}{2}}$ in the last inequality.

\subsection{Assumption 3:}
We consider the last case that $l=k-1$ and $2\leq k\leq n-2$. The main difference between the last two cases is $VI$ of \eqref{hess6}.

By using Lemma \ref{HS1} and Lemma \ref{f11est} we have
\begin{equation}
    \begin{aligned}\label{7,k-1}
      -\sum_{\gamma,\g\geq2}\lambda_{\gamma}\ddot{F}^{\gamma\gamma\g\g}\lambda_{\g}\geq C\k_1\cdots\k_{k-2}\left|[\lambda]_{1,\dots,k-2}^{\perp}\right|^2\geq C\k_1\k_2\cdots\k_{k-2}|\lambda^l|^2\geq C_7\dot{F}^{11}\k_1\sum_{p=k}^n\lambda_p^2.
    \end{aligned}
\end{equation}

Note that according to the assumption in this case,  for each $k\leq p\leq n$
\begin{align}\label{fpp,k-1}
   \f1-\frac{C(A)}{N_2^2}\r\k_1\cdots\k_{k-1} \leq\dot{F}^{pp}\leq \f1+\frac{C(A)}{N_2^2}\r\k_1\cdots\k_{k-1}.
\end{align} 

Recall \eqref{b4}
\begin{equation}
 \begin{aligned}
\sum_{i=2}^{k-1}\ddot{F}^{11ii}\lambda_{i}+\sum_{p=k}^n\ddot{F}^{11pp}\lambda_{p}\geq\sum_{i=2}^l\f\frac{\dot{F}^{ii}\k_i}{\k_1^2}\lambda_{i}-c\frac{\dot{F}^{11}}{\k_1\k_i}|\lambda_i|\r+\sum_{p=l+1}^n\f\frac{\dot{F}^{pp}\k_p}{\k_1^2}\lambda_{p}-c\frac{\dot{F}^{11}}{\k_1}|\lambda_p|\r.
\end{aligned}   
\end{equation}

Then by \eqref{perp}, \eqref{kifii} and \eqref{fpp,k-1},
\begin{align}\label{lp,k-1}
 2\left|\sum_{k\leq p\leq n} \frac{\dot{F}^{pp}}{\k_1^2}ah_{111}\k_p\lambda_{p}\right|\leq&\frac{C}{N_2^2}\frac{\k_1\cdots\k_{k-1}|ah_{111}\lambda_p|}{\k_1^2}\leq \frac{C'}{N_2^2}\dot{F}^{11}|ah_{111}\lambda_p|\nonumber\\
 \leq &\frac{C}{N_2^2}\dot{F}^{11}\k_1 \sum_{k\leq p\leq n}\lambda_p^2+\frac{Ca^2}{ N_2^2}\frac{\dot{F}^{11}h_{111}^2}{\k_1}.
\end{align}
Similar to \eqref{b4}, we have from \eqref{b1} and \eqref{lp,k-1}
\begin{equation}\label{b4,2}
 \begin{aligned}
2ah_{111}\sum_{\gamma=2}^n\ddot{F}^{11\gamma\gamma}\lambda_{\gamma}
\geq &2ah_{111}\sum_{i=2}^l\frac{\dot{F}^{11}}{\k_i}\lambda_{i}-8\epsilon(N_2)\frac{\dot{F}^{11}|ah_{111}\lambda_{i}|}{\k_i}-2c\sum_{i=2}^l \frac{\dot{F}^{11}}{\k_1\k_i}|ah_{111}\lambda_i|\\
&-\frac{C}{N_2^2}\dot{F}^{11}\k_1 \sum_{k\leq p\leq n}\lambda_p^2-\frac{Ca^2}{ N_2^2}\frac{\dot{F}^{11}h_{111}^2}{\k_1}-2c\frac{\dot{F}^{11}|ah_{111}\lambda_p|}{\k_1}\\
\geq&2a\sum_{i=2}^l\frac{\dot{F}^{11}h_{111}}{\k_i}\lambda_{i}-\frac{C'(a^2+1)}{N_2^2}\frac{\dot{F}^{11}h_{111}^2}{\k_1}-\frac{C+c}{N_2^2}\dot{F}^{11}\k_1 \sum_{k\leq p\leq n}\lambda_p^2\\
&-ca^2 N_2^2\frac{\dot{F}^{11}h_{111}^2}{\k_1^3}-\frac{Ca^2}{ N_2^2}\frac{\dot{F}^{11}h_{111}^2}{\k_1}.
\end{aligned}   
\end{equation}
where we use \eqref{lambdai} and Cauchy-Schwartz inequality in the last inequality.

Recall \eqref{b22}

\begin{equation}
\begin{aligned}
    -2\frac{b}{|\b|^2}\sum_{p,\gamma\geq2}\dot{F}^{pp}\ddot{F}^{pp\gamma\gamma}\lambda_{\gamma}= 
    &-2\frac{b}{|\b|^2}\sum_{l+1\leq p\leq n,2\leq i\leq l}\dot{F}^{pp}\f\frac{\dot{F}^{pp}}{\k_i}-\frac{\s_{k-1}(\k|ip)}{\k_i}\r\lambda_i\\
    & -2\frac{b}{|\b|^2}\sum_{2\leq j\neq i\leq l}\dot{F}^{jj}\ddot{F}^{jj,ii}\lambda_i-2\frac{b}{|\b|^2}\sum_{2\leq\gamma\leq n,l+1\leq q\leq n}\dot{F}^{\gamma\gamma}\ddot{F}^{\gamma\gamma,qq}\lambda_q.
\end{aligned}    
\end{equation}
Similar to \eqref{6,21}, we have from \eqref{lambdai}
\begin{equation}\label{6,21,2}
    -2\frac{b}{|\b|^2}\sum_{2\leq j\neq i\leq l}\dot{F}^{jj}\ddot{F}^{jj,ii}\lambda_i\geq -2\sum_{2\leq j\neq i\leq l}C(|a|+1)\frac{\dot{F}^{11}|h_{111}\lambda_i|}{\k_i\k_j^3}\geq-C'(a^2+1)\frac{\dot{F}^{11}h_{111}^2}{\k_1\k_j^3}. 
\end{equation}
Different from \eqref{b22}, we have by \eqref{betalowerbound1}
\begin{equation}
    \dot{F}^{\gamma\gamma,pp}\leq C(n,k)\k_1\cdots\k_{k-2}\leq \frac{C'}{\k_{k-1}}|\b|\leq\frac{C'}{N_2^2}|\b|.
\end{equation}
Hence similar to \eqref{6,22}, by \eqref{b22} we have 
 \begin{equation}\label{6,22,2}
 \begin{aligned}
      -2\frac{b}{|\b|^2}\sum_{2\leq\gamma\leq n,k\leq q\leq n}\dot{F}^{\gamma\gamma}\ddot{F}^{\gamma\gamma,qq}\lambda_q\geq& -2\sum_{l+1\leq q\leq n}\frac{C}{N_2^2}(|a|+1) \dot{F}^{11}|h_{111}\lambda_p|\\
    \geq &-\frac{C}{N_2^2}\sum_{p=l+1}^n\dot{F}^{11}\k_1\lambda_p^2-\frac{C(a^2+1)}{N_2^2}\frac{\dot{F}^{11}h_{111}^2}{\k_1}.   
 \end{aligned}
\end{equation}
Combining \eqref{6,24}, \eqref{6,21,2} and \eqref{6,22,2}, we have
\begin{equation}\label{6,24,2}
\begin{aligned}
    -2\frac{b}{|\b|^2}\sum_{p,\gamma\geq2}\dot{F}^{pp}\ddot{F}^{pp\gamma\gamma}\lambda_{\gamma}
    \geq &-2 a\frac{\dot{F}^{11}h_{111}\lambda_i}{\k_i}-\frac{C}{N_2^2}\sum_{p=l+1}^n\dot{F}^{11}\k_1\lambda_p^2-\frac{C(a^2+1)}{N_2^2}\frac{\dot{F}^{11}h_{111}^2}{\k_1}.
\end{aligned}    
\end{equation}
By \eqref{b4,2} and \eqref{6,24,2}, similar to \eqref{46} we obtain 
\begin{align}\label{46,k-1}
&2ah_{111}\sum_{\gamma=2}^n\ddot{F}^{11\gamma\gamma}\lambda_{\gamma}-2\frac{b}{|\b|^2}\sum_{p,\gamma\geq2}\dot{F}^{pp}\ddot{F}^{pp\gamma\gamma}\lambda_{\gamma}\nonumber\\
\geq& -\frac{C}{N_2^2}\sum_{p=l+1}^n\dot{F}^{11}\k_1\lambda_p^2-\frac{C(a^2+1)}{N_2^2}\frac{\dot{F}^{11}h_{111}^2}{\k_1}.
\end{align}

For $3\leq k\leq n-2$, combining \eqref{a1}, \eqref{d}, \eqref{711}, \eqref{f1pfp}, \eqref{7,k-1} and \eqref{46,k-1}, similar to \eqref{hess7}, \eqref{hess6} becomes
\begin{equation}\label{hess,k-1}
  \begin{aligned}
      -\ddot{F}^{pp,qq}\xi_p\xi_q+&\f 2-\frac{C(A)}{\k_1}\r\sum_{p\neq 1}\frac{F^{pp}\xi_p^2}{\k_1}\\
      \geq
&(1+a)^2\frac{k(k-1)Fh_{111}^2}{\k_1^2}- C(a^2+1)\frac{\dot{F}^{11}h_{111}^2}{\k_1^2}+2a^2\frac{\dot{F}^{11}h_{111}^2}{\k_1}-\frac{C(a^2+1)}{N_2^2}\frac{\dot{F}^{11}h_{111}^2}{\k_1}\\
&-\frac{C_8}{N_2^2}\sum_{p=k}^n\dot{F}^{11}\k_1\lambda_p^2+C_7\dot{F}^{11}\k_1\sum_{p=k}^n\lambda_p^2+2\f1-\frac{c}{\k_1}\r\sum_{p=k}^n\frac{F^{pp}\xi_p^2}{\k_1}\\
\geq
&2a^2\frac{\dot{F}^{11}h_{111}^2}{\k_1}-\frac{C'(a^2+1)}{N_2^2}\frac{\dot{F}^{11}h_{111}^2}{\k_1}+\frac{C_7}{2}\dot{F}^{11}\k_1^2\sum_{p=k}^n\lambda_p^2+2\f1-\frac{c}{\k_1}\r\sum_{k\leq p\leq n}\frac{F^{pp}\xi_p^2}{\k_1}
  \end{aligned}  
\end{equation}
by choosing $N_2>\sqrt{\frac{2C_8}{C_7}}$. \eqref{hess,k-1} implies that there exists some constant $C>0$ such that
\begin{align}\label{lamp,k-1}
    |\lambda_p|\leq \frac{C(|a|+1)|h_{111}|}{\k_1} 
\end{align}

Now we consider the eighth term in \eqref{hess6}. Using \eqref{perp} and \eqref{fpp,k-1}, we obtain similar to \eqref{xi,2}

\begin{align}\label{8,k-1}
 2\sum_{k\leq p\leq n}F^{pp}\xi_p^2=&2\sum_{k\leq p\leq n}F^{pp}\f(1+a)\frac{\k_ph_{111}}{\k_1}+\frac{b\dot{F}^{pp}}{|\b|^2}+\lambda_p\r^2\nonumber\\
 \geq &2(a+1)^2\sum_{k\leq p\leq n}\frac{\dot{F}^{pp}\k_p^2h_{111}^2}{\k_1^2}-C(a^2+1)\frac{\dot{F}^{11}h_{111}^2}{\k_1^2}\\
 &+4(1+a)\sum_{k\leq p\leq n}\frac{(\s_{k-1}-\k_p\s_{k-2}(\k|p))\k_ph_{111}\lambda_p}{\k_1}\nonumber\\
 \geq &2(a+1)^2\frac{\dot{F}^{pp}\k_p^2h_{111}^2}{\k_1^2}-\frac{2C}{N_2^2}(|a|+1)\dot{F}^{11}\k_1|\lambda_p h_{111}|-C(a^2+1)\frac{\dot{F}^{11}h_{111}^2}{\k_1^2}\nonumber\\
 \geq &2(a+1)^2\frac{\dot{F}^{pp}\k_p^2h_{111}^2}{\k_1^2}-\frac{C_9}{N_2^2}\dot{F}^{11}\k_1^2\lambda_p^2-\frac{2C}{N_2^2}(a^2+1)\dot{F}^{11}h_{111}^2
\end{align}
where we use \eqref{perp} in the second inequality.
Hence combining \eqref{74} and \eqref{lamp,k-1}, $VII$  of \eqref{hess6} becomes
\begin{equation}\label{hess9}
   \begin{aligned}
    2\f 1-\frac{c}{\k_1}\r\sum_{k\leq p\leq n}\frac{F^{pp}\xi_p^2}{\k_1}
    \geq&2\f1-\frac{c}{\k_1}\r(a+1)^2\sum_{k\leq p\leq n}\frac{\dot{F}^{pp}\k_p^2h_{111}^2}{\k_1^3}-\frac{C_9}{N_2^2}\dot{F}^{11}\k_1\sum_{k\leq p\leq n}\lambda_p^2\\
    &-\frac{2C}{N_2^2}(a^2+1)\frac{\dot{F}^{11}h_{111}^2}{\k_1}\\
\geq&4(1+a)^2\frac{\dot{F}^{11}h_{111}^2}{\k_1}-\frac{C_9}{N_2^2}\dot{F}^{11}\k_1\sum_{k\leq p\leq n}\lambda_p^2-\frac{C'(a^2+1)}{N_2^2}\frac{\dot{F}^{11}h_{111}^2}{\k_1}
\end{aligned} 
\end{equation}
Plugging \eqref{hess9} into \eqref{hess,k-1}, we obtain
\begin{equation}
   \begin{aligned}
  -\ddot{F}^{pp,qq}\xi_p\xi_q+\f 2-\frac{C(A)}{\k_1}\r\sum_{p\neq 1}\frac{F^{pp}\xi_p^2}{\k_1}\geq&
(2a^2+4(1+a)^2-\frac{C_{10}(a^2+1)}{N_2^2})\frac{\dot{F}^{11}h_{111}^2}{\k_1}\\
\geq &\f \frac{3}{2}a^2+4(1+a)^2-\frac{C_{10}}{N_2^2}+\f\frac{1}{2}-\frac{C_{10}}{N_2^2}\r a^2\r\frac{\dot{F}^{11}h_{111}^2}{\k_1}\\
\geq& \f\frac{12}{11}--\frac{C_{10}}{N_2^2}+\f\frac{1}{2}-\frac{C_{10}}{N_2^2}\r a^2\r\frac{\dot{F}^{11}h_{111}^2}{\k_1}\\
>&\frac{\dot{F}^{11}h_{111}^2}{\k_1}
\end{aligned} 
\end{equation}
by choosing $N_2>\max\lbrace\sqrt{\frac{4C_9}{C_7}},\sqrt{11C_{10}}\rbrace$.
For $k=2$, combining \eqref{a1},\eqref{d}, \eqref{a2}, \eqref{7,k-1} \eqref{b4,2}, and \eqref{6,22,2}, \eqref{hess6} implies that
\begin{equation}\label{hess,k-1,2}
  \begin{aligned}
      -\ddot{F}^{pp,qq}\xi_p\xi_q+&\f 2-\frac{C(A)}{\k_1}\r\sum_{p\neq 1}\frac{F^{pp}\xi_p^2}{\k_1}\\
      \geq
&(1+a)^2\frac{k(k-1)Fh_{111}^2}{\k_1^2}- C(a^2+1)\frac{\dot{F}^{11}h_{111}^2}{\k_1^2}+2a\f a-(1+a)\frac{kF}{\k_1\dot{F}^{11}}\r\frac{\dot{F}^{11}h_{111}^2}{\k_1}\\
&-\frac{C(a^2+1)}{N_2^2}\frac{\dot{F}^{11}h_{111}^2}{\k_1}
-\frac{C_8}{N_2^2}\sum_{p=k}^n\dot{F}^{11}\k_1\lambda_p^2+C_7\dot{F}^{11}\k_1\sum_{p=k}^n\lambda_p^2+2\f1-\frac{c}{\k_1}\r\sum_{p=k}^n\frac{F^{pp}\xi_p^2}{\k_1}\\
\geq
&\f2a^2+\f(1+a)^2k(k-1)-2a(1+a)k\r\frac{F}{\k_1\dot{F}^{11}}\r\frac{\dot{F}^{11}h_{111}^2}{\k_1}\\
&-\frac{C_9(a^2+1)}{N_2^2}\frac{\dot{F}^{11}h_{111}^2}{\k_1}+\frac{C_7}{2}\dot{F}^{11}\k_1^2\sum_{p=k}^n\lambda_p^2+2\f1-\frac{c}{\k_1}\r\sum_{k\leq p\leq n}\frac{F^{pp}\xi_p^2}{\k_1}
  \end{aligned}  
\end{equation}
by choosing $N_2>\sqrt{\frac{2C_8}{C_7}}$.
Using Proposition \ref{sigma_kkappa1negative}, we obtain that either
\begin{align*}
     \textbf{C}:=2a^2+\f(1+a)^2k(k-1)-2a(1+a)k\r\frac{F}{\k_1\dot{F}^{11}}\geq 2a^2
\end{align*}
if the coefficient of $F/\k_1\dot{F}_{11}$ is non-negative or
\begin{equation}\label{coe,2}
   2a^2+(2-2a^2)(1-\frac{\gamma}{2})= 2+(a^2-1)\gamma 
\end{equation}
 and $a^2>1$ meanwhile. 
     
If the coefficient of $F/\k_1\dot{F}_{11}$ is negative, \eqref{hess,k-1,2} becomes 
      \begin{equation}
          \begin{aligned}
        -\ddot{F}^{pp,qq}\xi_p\xi_q+\f 2-\frac{C(A)}{\k_1}\r\sum_{p\neq 1}\frac{F^{pp}\xi_p^2}{\k_1}
        \geq &\f2+(a^2-1)\gamma-2C_{10}N_2^{-2}-C_{10}(a^2-1)N_2^{-2}\r\frac{\dot{F}^{11}h_{111}^2}{\k_1}\\
 \geq &\frac{\dot{F}^{11}h_{111}^2}{\k_1}.
      \end{aligned}   
      \end{equation}
      by assuming $N_2\geq \max\lbrace \sqrt{\frac{C_9}{\gamma}}, \sqrt{2C_{10}}\rbrace $,
      which implies the \textbf{main claim}.
      
Using \eqref{sigmakkappapsqure} and \eqref{8,k-1}, we have
      \begin{align}\label{hess10}
 2\sum_{k\leq p\leq n}F^{pp}\xi_p^2
 \geq &2(a+1)^2\frac{\dot{F}^{pp}\k_p^2h_{111}^2}{\k_1^2}-\frac{C_{11}}{N_2^2}\dot{F}^{11}\k_1^2\lambda_p^2-\frac{2C}{N_2^2}(a^2+1)\dot{F}^{11}h_{111}\nonumber\\
 \geq &4(a+1)^2\frac{\f \k_1(\k_1\dot{F}^{11}-\s_k)\r h_{111}^2}{\k_1^2}-\frac{C_{11}}{N_2^2}\dot{F}^{11}\k_1^2\lambda_p^2\\
 &-\frac{C'}{N_2^2}(a^2+1)\dot{F}^{11}h_{111}^2.
\end{align}
Then by choosing $N_2\geq \sqrt{\frac{2C_11}{C_7}}$, \eqref{hess,k-1,2} becomes 
 \begin{align}
      -\ddot{F}^{pp,qq}\xi_p\xi_q&+2\f1-\frac{c}{\k_1}\r\sum_{p\neq 1}\frac{F^{pp}\xi_p^2}{\k_1}\nonumber\\
\geq&\f2a^2+4(1+a)^2+\f2(1+a)^2-4a(1+a)-4(1+a)^2\r\frac{F}{\k_1\dot{F}^{11}}\r\frac{\dot{F}^{11}h_{111}^2}{\k_1}\nonumber\\
      &-\frac{C_{12}}{N_2^2}(a^2+1)\frac{\dot{F}^{11}h_{111}^2}{\k_1}.
      \end{align}

    For $k=2$, if
      \[2(1+a)^2-4a(1+a)-4(1+a)^2<0,\] then $a<-1$ or $a>-\frac{1}{3}$ and 
      \eqref{hess10} implies
      \begin{align}
      -\ddot{F}^{pp,qq}\xi_p\xi_q&+2\f1-\frac{c}{\k_1}\r\sum_{p\neq 1}\frac{F^{pp}\xi_p^2}{\k_1}\nonumber\\
\geq&\f2a^2+4(1+a)^2+\f2(1+a)^2-4a(1+a)-4(1+a)^2\r(1-\frac{\gamma}{2})\r\frac{\dot{F}^{11}h_{111}^2}{\k_1}\nonumber\\
      &-\frac{C_{12}}{N_2^2}(a^2+1)\frac{\dot{F}^{11}h_{111}^2}{\k_1}\nonumber\\
      \geq&\f (3a^2+4a+1)\gamma +2\r\frac{\dot{F}^{11}h_{111}^2}{\k_1}--\frac{C_{12}}{N_2^2}(a^2+1)\frac{\dot{F}^{11}h_{111}^2}{\k_1}\nonumber\\
      \geq& \f a^2\gamma'+2(a+1)^2\gamma'+2(1-\frac{\gamma'}{2})\r\frac{\dot{F}^{11}h_{111}^2}{\k_1}-\frac{C_{12}}{N_2^2}(a^2+1)\frac{\dot{F}^{11}h_{111}^2}{\k_1}\\
      > &\frac{\dot{F}^{11}h_{111}^2}{\k_1}.
      \end{align}
 Here we denote by $\gamma'=\min\left\lbrace \frac{1}{2},\gamma\right\rbrace$ and choose $\k_1\geq \max\left\lbrace N_2^{4k},\f\frac{C_7}{\gamma'}\r^{2k}\right\rbrace$.  
If  \[2(1+a)^2-4a(1+a)-4(1+a)^2\geq0,\] 
\eqref{hess10} implies   
 \begin{align}
      -\ddot{F}^{pp,qq}\xi_p\xi_q&+2\f1-\frac{c}{\k_1}\r\sum_{p\neq 1}\frac{F^{pp}\xi_p^2}{\k_1}\nonumber\\
    \geq &\f2a^2+4(1+a)^2\r\frac{\dot{F}^{11}h_{111}^2}{\k_1}-\frac{C_{12}}{N_2^2}(a^2+1)\frac{\dot{F}^{11}h_{111}^2}{\k_1}\nonumber\\
    \geq &\f \frac{3}{2}a^2+4(1+a)^2-\frac{C_{12}}{N_2^2}+\f\frac{1}{2}-\frac{C_{12}}{N_2^2}\r a^2\r\frac{\dot{F}^{11}h_{111}^2}{\k_1}\\
    > &\frac{\dot{F}^{11}h_{111}^2}{\k_1}
    \end{align}
     by choosing $N_2\geq  \sqrt{11C_7}$. 
     
Combining all the cases under \textbf{Assumption 2} and Proposition \ref{l,N}, we obtain there exists some positive constant $C$ and $c$ only depending on $n,k,A,\s_k,a_1, \d_0$ or equivalently $n,k,\s_k,A$ such that the \textbf{main claim} holds as long as $N_2>\max\lbrace C,1\rbrace$ and $\k_1>CN_2^{2^{k-1}}$. Thus we complete the proof of Theorem \ref{main}.

\end{document}